\documentclass[11pt,leqno]{amsart}

\usepackage{graphicx}
\graphicspath{ {./images/} }

\usepackage[dvipsnames]{xcolor}
\definecolor{b}{HTML}{4472c4}
\definecolor{o}{HTML}{ED7D31}
\definecolor{g}{HTML}{70ad47}

\usepackage{fullpage}
\usepackage{amsmath,amsthm,amssymb,latexsym,color}
\usepackage[hidelinks,unicode,linktocpage=true]{hyperref}
\usepackage{lipsum}
\usepackage{enumitem,bbm}
\usepackage{hyperref}
\hypersetup{
    colorlinks=true,
    linkcolor=blue,
    filecolor=blue,
    citecolor=blue,
    urlcolor=cyan}

\usepackage{tikz}
\usepackage{subfigure}

\usepackage{multirow}

\newtheorem{theorem}{Theorem}[section]

\newtheorem{lemma}[theorem]{Lemma}
\newtheorem{corollary}[theorem]{Corollary}
\newtheorem{definition}[theorem]{Definition}
\newtheorem{proposition}[theorem]{Proposition}
\newtheorem{fact}[theorem]{Fact}
\newtheorem{claim}[theorem]{Claim}

\newtheorem{problem}[theorem]{Problem}

\theoremstyle{remark}
\newtheorem{remark}[theorem]{Remark}

\theoremstyle{plain}

\def\dist{{\rm dist}}
\def\diam{{\rm diam}}

\def\dist{{\rm dist}}

\newcommand{\seminorm}[1]{{\left\vert\kern-0.25ex\left\vert\kern-0.25ex\left\vert #1
    \right\vert\kern-0.25ex\right\vert\kern-0.25ex\right\vert}}

\begin{document}

\title[Discrete Poincar\'e inequalities and universal approximators for random graphs]{Discrete Poincar\'e inequalities and universal approximators for random graphs}

\author{Dylan J. Altschuler, Pandelis Dodos, Konstantin Tikhomirov and Konstantinos Tyros}

\address{Department of Mathematical Sciences, Carnegie Mellon University}
\email{daltschu@andrew.cmu.edu}

\address{Department of Mathematics, University of Athens, Panepistimiopolis 157 84, Athens, Greece}
\email{pdodos@math.uoa.gr}

\address{Department of Mathematical Sciences, Carnegie Mellon University}
\email{ktikhomi@andrew.cmu.edu}

\address{Department of Mathematics, University of Athens, Panepistimiopolis 157 84, Athens, Greece}
\email{ktyros@math.uoa.gr}



\begin{abstract}
Nonlinear Poincar\'e inequalities are indispensable tools in the study of dimension reduction and low-distortion embeddings of graphs into metric spaces, and have found remarkable algorithmic applications. A basic open problem, posed by Jon Kleinberg (2013), asks whether the optimal nonlinear Poincar\'e constant for maps between two independent $3$-regular random graphs is dimension-free, i.e., independent of vertex-set sizes. We give a complete and affirmative resolution to Kleinberg's problem, also allowing for arbitrary graph degrees. As a corollary, we obtain a stochastic construction of $O(1)\text{-universal}$ approximators for random graphs, answering a question of Mendel and Naor.
\end{abstract}

\maketitle

\tableofcontents

\numberwithin{equation}{section}

\section{Introduction} \label{sec1}

\subsection{Nonlinear spectral gaps}

Let $d\geqslant 3$ be an integer, let $G=(V_G,E_G)$ be a $d$-regular graph, set $n:=|V_G|$, and write the eigenvalues of the adjacency matrix $A_G$ of $G$ in increasing order as
\begin{equation} \label{eigen-e1}
\lambda_n(G)\leqslant \cdots \leqslant \lambda_2(G)\leqslant \lambda_1(G)=d.
\end{equation}
It is a standard observation that the quantity $\frac{d}{2(d-\lambda_2(G))}$ (essentially, the reciprocal of the spectral gap) can also be defined as the best constant $\gamma\in (0,\infty]$ such that for any $f\colon V_G\to \mathbb{R}^k$,
\begin{equation} \label{eq-poincare-e1}
\frac{1}{|V_G|^2} \sum_{v,u\in V_G} \big\|f(v)-f(u)\big\|_2^2 \leqslant
\frac{\gamma}{|E_G|} \sum_{\{v,u\}\in E_G} \big\|f(v)-f(u)\big\|_2^2.
\end{equation}
Nonlinear (i.e., non-Euclidean) Poincar\'{e} inequalities are extensions of~\eqref{eq-poincare-e1} to other norms and, more generally, functions taking values in metric spaces. Such inequalities have proven indispensable in the study of embeddings and metric dimension reduction, and hold a role of central importance in a number of areas, especially metric geometry and algorithm design.

Some of the earlier appearances of nonlinear Poincar\'e inequalities came from the study of $\text{bi-Lipschitz}$ embeddings, such as the works of Enflo \cite{E76}, Bourgain--Milman--Wolfson \cite{BMW86}, Gromov \cite{Gr83}, Pisier \cite{Pi86}, Linial--London--Rabinovich \cite{LLR95}, and Matou\v{s}ek \cite{Ma97}.

Gromov \cite{Gr03} initiated the usage of nonlinear Poincar\'{e} inequalities in the context of coarse embeddings, with connections to the Novikov conjecture. This line of work was further investigated by Ozawa \cite{Oz04}, Kasparov--Yu \cite{KY06}, and Pisier \cite{Pi10}, eventually leading to Lafforgue's~\cite{La08} construction of \emph{superexpanders}, regular graphs with remarkable coarse non-embeddability properties. Alternate constructions of superexpanders were also given by $\text{Mendel--Naor}$ \cite{MN14}, and $\text{de Laat--de la Salle}$~\cite{dLdS23}.

Variants of nonlinear Poincar\'e inequalities have also been used fruitfully in algebraic settings; we refer to the work of Naor--Silberman \cite{NS11} and the references therein.

More recently, nonlinear Poincar\'{e} inequalities have found spectacular algorithmic applications. We refer the reader to the works of Andoni--Naor--Nikolov--Razenshteyn--Waingarten \cite{ANNRW18} and $\text{Mendel--Naor \cite{MN15}}$, as well as the survey paper of Eskenazis \cite{Es22}, for an overview of these exciting developments.

A systematic study of nonlinear versions of \eqref{eq-poincare-e1} was undertaken in a series of papers \cite{MN13,MN14,Na14,MN15} by Mendel and Naor. The following definition---which first appeared in \cite{MN14}---conceptualizes and unifies the aforementioned nonlinear Poincar\'{e} inequalities.

\begin{definition}[Nonlinear spectral gap]
Let $d\geqslant 3$ be an integer, and let $G=(V_G,E_G)$ be a $d\text{-regular}$ graph. Also let $\mathcal{M}=(M,\varrho)$ be a metric space, and let $p\geqslant 1$. By $\gamma(G,\varrho^p)$ we denote the smallest constant $\gamma\in (0,\infty]$ such that for any $f\colon V_G \to M$,
\begin{equation} \label{eq-poincare}
\frac{1}{|V_G|^2} \sum_{v,u\in V_G} \varrho\big(f(v),f(u)\big)^p \leqslant
\frac{\gamma}{|E_G|} \sum_{\{v,u\}\in E_G} \varrho\big(f(v),f(u)\big)^p.
\end{equation}
\end{definition}

\subsection{Main results}

This article investigates nonlinear spectral gaps when the metric space~$\mathcal{M}$ is the vertex set of a uniformly random regular graph $H$ equipped with the shortest-path distance $\dist_H$. In addition to being a natural setting with intrinsic interest---random regular graphs are ubiquitous in mathematics and theoretical computer science---there are also significant motivating applications that will be discussed shortly.

\subsubsection{ \! } \label{subsubsec1.2.1}

A basic open question in this context, asked by Jon Kleinberg \cite[Question 2.4]{MN15}, is the following.

\begin{problem}[Kleinberg] Let $G$ and $H$ be two independent and uniformly random 3-regular graphs. Is the Poincar\'{e} constant $\gamma(G,\dist_H^2)$ bounded by a universal constant, independent of $|V_G|$ and $|V_H|$, with high probability?
\end{problem}

Our first main result resolves this question affirmatively with significant added generality. In particular, our result holds when $G$ and $H$ have arbitrary (possibly different) degrees, and it upper bounds $\gamma(G,\dist_H^p)$ for any $p \geqslant 1$, independently of the degree of $H$.

For the remainder of this article, for any pair $n\geqslant d\geqslant 3$ of integers, $G(n,d)$ denotes the set of all $d$-regular graphs on $[n]:=\{1,\dots,n\}$, and $\mathbb{P}_{G(n,d)}$ denotes the uniform probability measure on~$G(n,d)$; we will always assume $dn$ is even.

\begin{theorem}[Affirmative solution of Kleinberg's problem] \label{thm-kleinberg}
Let $d, \Delta\geqslant 3$ be integers, and for every~$p\geqslant 1$, set
\begin{equation} \label{klein-e1}
\Gamma(d,p):= \exp\Big( 10^{12}\, 4^p\, \log^2 d\Big).
\end{equation}
Then\footnote{See Section \ref{sec2} for details concerning our asymptotic notation.},
\begin{equation} \label{klein-e2}
\mathbb{P}_{G(n,d)}\times \mathbb{P}_{G(m,\Delta)} \Big[(G,H)\colon \gamma(G,\dist_H^p)\leqslant \Gamma(d,p)
\text{ for every } p\geqslant 1\Big] \geqslant 1-O_{d,\Delta}\Big(\frac{1}{\min\{n,m\}^\tau}\Big),
\end{equation}
where $\tau=\tau(d,\Delta)>0$.
\end{theorem}

\begin{remark} \label{rem-new}
A principle reason for the significant and enduring interest in Kleinberg's problem is that it represents a basic meta-problem of the field: the absence of probabilistic and combinatorial tools for reasoning about nonlinear spectral gaps. See, e.g., \cite{MN15,ADTT24,EMN25a} for discussion of this perspective. Our proof is combinatorial, directly addressing this gap. The applications in this paper can also be contextualized as progress in this direction.
\end{remark}

Letting $\mathcal{M}=(M,\varrho_{\mathcal{M}})$ and $\mathcal{N}=(N,\varrho_{\mathcal{N}})$ be metric spaces, recall that the \emph{$($bi-Lipschitz$)$ distortion of $\mathcal{M}$ into $\mathcal{N}$}, denoted by $c_{\mathcal{N}}(\mathcal{M})$, is defined as the smallest constant $D>0$ for which there~exists a map\footnote{Note that this map $e$ is necessarily injective.} $e\colon M\to N$ and (a scaling factor) $s>0$ such that $s \varrho_{\mathcal{M}}(i,j)\leqslant \varrho_{\mathcal{N}}\big(e(i),e(j)\big)\leqslant s D \varrho_{\mathcal{M}}(i,j)$ for all $i,j\in M$. Theorem \ref{thm-kleinberg} together with a standard application of~\eqref{eq-poincare}---reproduced in Subsection~\ref{subsec6.4} for the reader's convenience---yield that independent random regular graphs, viewed as metric spaces, have mutually incompatible geometries.

\begin{corollary}[Bi-Lipschitz distortion between random regular graphs] \label{cor-bi-Lip}
Let $d, \Delta\geqslant 3$ be integers. Then, for any pair of $m\geqslant n$ of positive integers,
\begin{equation} \label{bi-Lip-e1}
\mathbb{P}_{G(n,d)}\times \mathbb{P}_{G(m,\Delta)} \Big[(G,H)\colon c_H(G)\gtrsim_d \log n \Big] \geqslant 1-O_{d,\Delta}\Big(\frac{1}{n^\tau}\Big),
\end{equation}
where $\tau=\tau(d,\Delta)>0$ is as in Theorem \ref{thm-kleinberg}, and $c_H(G)$ denotes the bi-Lipschitz distortion of $(V_G,\dist_G)$ into $(V_H,\dist_H)$.
\end{corollary}

\subsubsection{\!} \label{subsubsec1.2.2}

Our second main result concerns sequences $(G_n)$ of regular graphs with degree $d\geqslant 3$ and $|V_{G_n}|\to\infty$, that satisfy, with high probability, a nonlinear Poincar\'{e} inequality with respect to a uniformly random regular graph $H$ with a Poincar\'{e} constant $\gamma(G_n,\dist_H^2)$ bounded independently of $n$ and the size of the vertex set and the degree of $H$.

The existence of a single sequence $(G_n)$ of $3$-regular graphs with this property is a celebrated result of Mendel and Naor \cite[Theorem~1.2]{MN15} that utilizes the nonlinear spectral calculus developed in~\cite{MN13,MN14}. In this work, Mendel and Naor mention the lack of tools for studying nonlinear spectral gaps of random graphs, and they allude to the natural question of whether their result can be achieved by a randomized construction. The following theorem affirms that this is the case; random regular graphs of any fixed degree $d\geqslant 3$ provide the desired alternative construction.

\begin{theorem}[Random version of the Mendel--Naor construction] \label{mn-random}
Let $d\geqslant 3$ be an integer, and for every $p\geqslant 1$, let\, $\Gamma(d,p)$ be as in \eqref{klein-e1}. There exist positive constants $C(d)$ and $\tau' = \tau'(d)$ with the following property. For every integer $n$ with $n \geqslant C(d)$ and $dn$ an even integer, there is a nonempty subset $\mathcal{A}_n$ of\, $G(n,d)$ with
\begin{equation} \label{mn-random-e1}
\mathbb{P}_{G(n,d)}\big[\mathcal{A}_n\big]\geqslant 1-O_d\Big(\frac{1}{n^{\tau'}}\Big),
\end{equation}
such that if\, $G_n\in\mathcal{A}_n$, then, for any integer~$\Delta\geqslant 3$,
\begin{equation} \label{mn-random-e2}
\mathbb{P}_{G(m,\Delta)} \Bigg[ H\colon \sup_{\substack{n\geqslant C(d) \\ dn \text{ even}}}
\gamma(G_n,\dist_H^p)\leqslant \Gamma(d,p)
\text{ for every } p\geqslant 1 \Bigg] \geqslant 1-O_{\Delta}\Big(\frac{1}{m^{\tau''}}\Big),
\end{equation}
where $\tau'' = \tau''(\Delta)>0$.
\end{theorem}

\begin{remark} \label{rem1}
As mentioned above, the existence of a sequence $(G_n)$ of $3$-regular graphs that satisfies~\eqref{mn-random-e2} for $p>1$ is due to Mendel and Naor \cite{MN15}. However, let us highlight that the case $p=1$ of Theorem \ref{mn-random} is a new result.  As noted in \cite[Remark 2.3]{MN15}, the type of martingale arguments used in \cite{MN13,MN14} typically fail for $p=1$. We bypass these analytical barriers by employing and further developing some of the combinatorial ideas introduced in~\cite{ADTT24}.

We also note that, in an independent work, Eskenazis, Mendel and Naor \cite[Theorem 1.20]{EMN25a} have succeeded in extrapolating the exponent in the original Mendel--Naor construction to cover the case of all~$p\geqslant 1$. Their argument uses the analytical tools developed in \cite{MN13,MN14,MN15} and, in the critical regime $p\in [1,2)$, they rely on geometric ideas---in particular, they utilize embeddings of random regular graphs into Euclidean cones over certain, well-behaved, metric spaces.
\end{remark}

\begin{remark} \label{rem2}
The core of the proofs of Theorems \ref{thm-kleinberg} and \ref{mn-random} consists in showing the case $p=1$. The case $p>1$ follows from the case $p=1$ and a version of Matou\v{s}ek's extrapolation argument for metric spaces taken from \cite{ADTT25}.
\end{remark}

\begin{remark} \label{rem3}
The constant $C(d)$ from Theorem \ref{mn-random} can be explicitly estimated, though obtaining a strong upper bound on $C(d)$ in terms of the degree $d$ would likely require revisiting some of the arguments in \cite{ADTT24}.
\end{remark}

\begin{remark} \label{rem4}
The bound $\Gamma(d,p)$ of the Poincar\'{e} constant in Theorems \ref{thm-kleinberg} and \ref{mn-random} is neither optimal, nor even optimized with respect to the current proof. Given the importance of this invariant, especially in algorithmic applications, it would be valuable to know the correct dependence of $\Gamma(d,p)$ in terms of the degree $d$ and the exponent $p$.
\end{remark}

\begin{problem}[Optimality of the Poincar\'{e} constant]
Determine the optimal dependence on $d,p$ of the bound $\Gamma(d,p)$ of the Poincar\'{e} constant. Can a significant numerical improvement be made to our result for small degrees---e.g., for $d=3$---and exponent $p=1$?
\end{problem}

\subsection{Application to universal approximators}

Given a metric space $\mathcal{M}=(M,\varrho)$, a constant $D>0$ and an exponent $p\geqslant 1$, we say, following \cite{BGS07}, that a multi-graph\footnote{We allow multi-graphs to have self-loops.} $U$ on $[k]$ is a \emph{$D\text{-universal}$ approximator with respect to $\varrho^p$} if there exists (a scaling factor) $s>0$ such that for every $x_1,\dots,x_k\in M$,
\[ \frac{1}{k^2} \sum_{i,j=1}^k \varrho(x_i,x_j)^p \leqslant \frac{s}{|E_U|}
\sum_{\{i,j\}\in E_U} \varrho(x_i,x_j)^p \leqslant
\frac{D}{k^2} \sum_{i,j=1}^k \varrho(x_i,x_j)^p. \]
Universal approximators for general metric spaces, as well as for specific subclasses of metric spaces, have been actively studied in the theoretical computer science literature and allied areas; we refer the reader to \cite[Section 2]{MN15} for a general overview of this line of research.

The first examples of sequences of multi-graphs $(U_k)$ with $V_{U_k}=[k]$ and $|E_{U_k}|=O(k)$ that are $O(1)\text{-universal}$ approximators with respect to random regular graphs and exponent $p>1$ have been constructed by Mendel and Naor \cite[Theorem 2.1]{MN15}. The case of $p = 1$ is the original setting considered for algorithmic applications \cite{In99}, and it was left as a central open problem in~\cite{MN15}.

As a consequence of the main results of our paper, we may construct universal approximators with respect to random regular graphs for any exponent $p\geqslant 1$, thereby also covering the endpoint case $p=1$.

\begin{theorem} \label{universal-approximator}
For every $p\geqslant 1$, let\, $\Gamma(p)=\Gamma(3,p)$ be as in \eqref{klein-e1}. Then, for every positive integer~$k$, there exists a multi-graph $U_k$ on $[k]$ with $|E_{U_k}|=O(k)$ such that, for any integer $\Delta\geqslant 3$,
\begin{align} \label{universal-e1}
\mathbb{P}_{G(m,\Delta)} \Big[H\colon \text{for } & \text{every $p\geqslant 1$ and every
integer $k\geqslant 1$, the multi-graph } U_k  \\
& \text{is a $\big(2^p\Gamma(p)\big)$-universal approximator of\, } \dist_H^p\Big]
\geqslant 1-O_{\Delta}\Big(\frac{1}{m^{\tau''}}\Big), \nonumber
\end{align}
where $\tau''=\tau''(\Delta)>0$ is as in Theorem \ref{mn-random}.
\end{theorem}

\begin{remark} \label{rem-newnew}
Universal approximators for random regular graphs---that is, the case $p=1$ in Theorem \ref{universal-approximator}---were also constructed by Eskenazis--Mendel--Naor \cite[Theorem~1.26]{EMN25a} (see, also, \cite{EMN25b}) as a consequence of their extrapolated version of the Mendel--Naor construction \cite[Theorem 1.2]{MN15}. The quantitative aspects of Theorem \ref{universal-approximator} are more effective when compared with \cite[Theorem~1.26]{EMN25a}; however, as we have already noted in Remark \ref{rem4}, our estimate for the constant $\Gamma(p)$ is, most likely, not optimal.
\end{remark}

\subsection{Proof overview}

We now briefly describe the key ideas of the proofs of Theorems \ref{thm-kleinberg} and~\ref{mn-random}. Let $G$ and $H$ be two independent random regular graphs on $[n]$ and $[m]$ with degrees $d$ and $\Delta$, respectively. Our goal is to prove that with high probability, for any $f\colon [n] \to [m]$,
\begin{equation} \label{over}
\frac{1}{n^2} \sum_{v,u \in [n]} \dist_H\big(f(v),f(u)\big) \leqslant
O_{d}\bigg( \frac{1}{|E_G|} \sum_{\{v,u\}\in E_G} \dist_H\big(f(v),f(u)\big)\bigg),
\end{equation}
where the implied constant is independent of $n, m, \Delta$, and $f$. As the starting point, let us discuss the extreme situations when $f$ is nearly injective, and when $f$ is nearly constant. Upon considering these cases, three natural mechanisms emerge as candidates for enabling the proof of \eqref{over}.

The first mechanism is spectral expansion. Consider the toy situation in which there are two points (``clusters'') in the image of $f$,  each of which has pre-image of size linear in $n$. As a consequence of spectral expansion of random graphs, with high probability there will be a linear number of edges in $G$ connecting the pre-images of the clusters. This implies \eqref{over}; however, the viability of this strategy rapidly deteriorates as the number of clusters begins to grow.

Second, we can leverage concentration of measure (namely, the fact that edges are nearly equidistributed) for random graphs. Consider the toy example in which $m \asymp |\mathrm{Image}(f)| \leqslant n$. As a deterministic consequence of the $\Delta$-regularity of $H$, for any fixed $x \in V_H$, there are at most order $m^{o(1)}$ vertices $y \in V_H$ with $\dist_H(x,y) = o(\log m)$. Then with very high probability (in particular, high enough to union bound over all possible choices of $f$), the graph $G$ has order $n$ edges between vertices which are mapped to logarithmic distance in $H$, yielding the validity of \eqref{over}. (A similar computation appears in \cite[Proposition 8.1]{MN15} to obtain \eqref{over} in the specific case that $m = n$ and $f$ is a bijection.)

Third, towards bridging the gap between these regimes, we utilize the locally tree-like behaviour of random graphs. A strong form of this property---developed by Arora et al. \cite{ALNRRV12} and Mendel--Naor \cite{MN15}---is that any sublinear-sized neighborhood of a random graph admits a $\text{low-distortion}$ embedding into $L_1$. This will be formalized in Definition \ref{def-R(e)} as \emph{property $\mathcal{R(\varepsilon)}$}, which asserts that the subgraph of $H$ induced by any subset of size at most $m^{1-\varepsilon}$---where $\varepsilon$ is an arbitrarily small constant---admits a bi-Lipschitz embedding into $L_1$. Property $\mathcal{R(\varepsilon)}$ is used in conjunction with a classical result of Matou\v{s}ek asserting that random graphs satisfy a Poincar\'e inequality with respect to $L_1$. Together, these facts eventually yield \eqref{over} for any $f$ with $|\mathrm{Image}(f)|\leqslant m^{1-\varepsilon}$.

In view of these tools, the critical regime for proving \eqref{over} is $m^{1-o(1)} \leqslant |\mathrm{Image}(f)| \leqslant o(m)$, where none of these mechanisms directly apply. We shall highlight three of the key ideas employed in this setting. First, we develop a random compression argument to ``reduce'' the image of $f$, thereby enabling the usage of property $\mathcal{R(\varepsilon)}$. Second, we adopt a ``cluster-level" view of $f$, in which vertices $v \in V_G$ are sorted by $\big|f^{-1}\big(f(v)\big)\big|$. This enables us to decompose the domain of $f$ into regions on which the function $f$ is nearly injective, or nearly constant. Third, our proof heavily utilizes a strong form of vertex expansion for random regular graphs---see Definition \ref{def-D(a)}---introduced in our previous work \cite{ADTT24}. This property is used in Lemma \ref{l5.8} as the key ingredient for understanding the statistical behaviour of the random compression scheme.

\subsection{Explicit constructions}

The aforementioned combinatorial property of regular graphs, presented in Definition \ref{def-D(a)} and referred to as \emph{property $\mathcal{D}(\alpha)$}, encodes spectral expansion as well as a strong form of vertex expansion. In view of the importance of this property to the current article and \cite{ADTT24}, we conclude by reiterating the following basic open problem.

\begin{problem}[\textcolor{black}{\cite[Problem 6.4]{ADTT24}}] \label{problem-explicit}
For all fixed $d\geqslant 3$, construct an explicit and deterministic sequence of $d$-regular graphs with property $\mathcal{D}(\alpha)$, for some $\alpha=\alpha(d)>0$ (see Definition~\ref{def-D(a)}).
\end{problem}

In this direction, important progress was recently made by Hsieh et al. \cite{HLMRZ25}. They constructed an explicit sequence of $d$-regular graphs such that the neighborhood of radius $t$ of any nonempty set $S$ of vertices, has growth rate of order $\big(d-1-o_d(1)\big)^t |S|$, where $o_d(1)$ is an error term that vanishes as the degree $d$ diverges. While this significantly improves over previous constructions, it still falls short from resolving Problem \ref{problem-explicit}.

\section{Background material} \label{sec2}

\subsection{Graph theoretic preliminaries}

All graphs in this paper are simple. If $G$ is any graph, then by $V_G$ we denote its vertex set, and by $E_G$ we denote the set of its edges; moreover, for any subset $S$ of $V_G$, by $G[S]$ we denote the induced on $S$ subgraph of $G$.

\subsubsection{Graph distance}

For any graph $G$, denote by $\dist_G(\cdot,\cdot)$ the shortest-path distance on~$G$ between vertices. By convention, set $\dist_G(v,u):=|V_G|$ if $v,u\in V_G$ are contained in different connected components of $G$. If $v\in V_G$ and $S\subseteq V_G$ is nonempty, then define the shorthand
\[ \dist_G(v,S):= \min\big\{ \dist_G(v,u)\colon u\in S\big\}. \]
Moreover, for every integer $\ell\geqslant 0$, let $B_G(S,\ell)$ denote the set $\{v\in V_G\colon \dist_G(v,S)\leqslant \ell\}$. (In particular, $B_G(S,0)=S$.) Finally, define the \emph{diameter} $\diam(G)$ of $G$ to be the quantity $\max\{\dist_G(v,u):v,u\in V_G\}$; namely, $\diam(G)$ is the diameter of the metric space~$(G,\dist_G)$.

\subsubsection{Regular graphs}

For any pair $n\geqslant d\geqslant 3$ of positive integers, $G(n,d)$ denotes the set of all $d$-regular graphs on $[n]:=\{1,\dots,n\}$, and $\mathbb{P}_{G(n,d)}$ denotes the uniform probability measure on $G(n,d)$.

\subsection{Asymptotic notation}

If $a_1,\dots,a_k$ are parameters, then we write $O_{a_1,\dots,a_k}(X)$ to denote a quantity that is bounded in magnitude by $X C_{a_1,\dots,a_k}$, where $C_{a_1,\dots,a_k}$ is a positive constant that depends on the parameters $a_1,\dots,a_k$. We also write $Y\lesssim_{a_1,\dots,a_k}\!X$ or $X\gtrsim_{a_1,\dots,a_k}\!Y$ for the estimate $|Y|=O_{a_1,\dots,a_k}(X)$. Finally, we write $Y=\Theta_{a_1,\dots,a_k}(X)$ if $Y=O_{a_1,\dots,a_k}(X)$ and $X=O_{a_1,\dots,a_k}(Y)$.

\subsection{Bi-Lipschitz distortion into $L_1$, and nonlinear spectral gaps}

Recall that for a metric space $\mathcal{M}=(M,\varrho)$, the (bi-Lipschitz) distortion of $\mathcal{M}$ into $L_1$ is denoted by $c_{L_1}(\mathcal{M})$. Since $L_1$ is a vector space, the quantity $c_{L_1}(\mathcal{M})$ can be equivalently defined as the smallest constant $D>0$ for which there exists a map $e\colon M\to L_1$ such that $\varrho(i,j)\leqslant \|e(i)-e(j)\|_{L_1}\leqslant D \varrho(i,j)$ for all $i,j\in M$. We will need the following basic fact relating the metric invariant $c_{L_1}(\mathcal{M})$ with nonlinear spectral gaps.

\begin{fact} \label{fact-L1}
Let $n\geqslant d\geqslant 3$ be integers, and let $G$ be a $d$-regular graph on $[n]$ with $\lambda_2(G)< d$. Also let $\mathcal{M}=(M,\varrho)$ be a finite metric space. Then,
\[ \gamma(G,\varrho) \leqslant c_{L_1}(\mathcal{M})\, \frac{d}{2\big(d-\lambda_2(G)\big)}. \]
\end{fact}
\begin{proof}
Since $\gamma(G,\|\cdot\|_{L_2}^2)=\frac{d}{2(d-\lambda_2(G))}$, Matou\v{s}ek's extrapolation argument \cite{Ma97} (see, also, \cite[Lemma 5.5]{BLMN05}) yields $\gamma(G,\|\cdot\|_{L_1})\leqslant \frac{d}{2(d-\lambda_2(G))}$. The result follows after invoking the definition of $c_{L_1}(\mathcal{M})$.
\end{proof}

\subsection{The Cheeger constant of regular graphs}

Let $d\geqslant 3$ be an integer, let $G$ be a $d$-regular graph, and recall that the \emph{Cheeger constant} of $G$ is defined by
\[ h(G):= \min_{\substack{\emptyset \neq S\subseteq V_G\\ |S|\leqslant \frac{|V_G|}{2}}}
\frac{\big|\big\{ \{v,u\}\in E_G:\; v\in S \text{ and } u\in S^{\complement}\big\}\big|}{|S|}. \]
We will need the well-known fact that the Cheeger constant of $G$ is related to the spectral properties of the adjacency matrix of $G$. Specifically, (see, e.g., \cite[Theorem 4.1]{HLW06}),
\begin{equation} \label{e-cheeger-e1}
\frac{d-\lambda_2(G)}{2} \leqslant h(G) \leqslant \sqrt{2d\big(d-\lambda_2(G)\big)}.
\end{equation}

\subsection{Extrapolation for metric spaces}

We will need a (one-sided) version of Matou\v{s}ek's extrapolation \cite{Ma97} that is applicable to general metric spaces. More precisely, we have the following theorem; see \cite[Theorem 4.5]{ADTT25} for a proof.

\begin{theorem} \label{extrapolation}
Let $d\geqslant 3$ be an integer, and let $G$ be a $d$-regular graph with $h(G)>0$. Also let $\mathcal{M}=(M,\varrho)$ be an arbitrary metric space. Then, for any $1\leqslant p \leqslant q <\infty$, we have
\begin{equation} \label{extra-e1}
\gamma(G,\varrho^q) \leqslant \max\bigg\{ \exp\Big( 64\cdot 4^q  \frac{d}{h(G)}  \log d\Big), \, 5^q \, 2^{\frac{q}{p}} \, \gamma(G,\varrho^p)^{\frac{q}{p}}\bigg\}.
\end{equation}
\end{theorem}

\section{Main technical results} \label{sec3}

We now detail the main technical results needed for the proofs of Theorems \ref{thm-kleinberg} and \ref{mn-random}.

\subsection{Combinatorial properties of regular graphs} \label{subsec3.1}

A key conceptual step is the isolation of certain combinatorial properties of regular graphs that are satisfied with high probability by random regular graphs.

The first property concerns the ``domain" graph $G$.

\begin{definition}[Property $\mathcal{D}(\alpha)$] \label{def-D(a)}
Let $n\geqslant d\geqslant 3$ be integers, and let $G$ be a $d$-regular graph on $[n]$. Given $0<\alpha \leqslant 1$, we say that $G$ satisfies \emph{property $\mathcal{D}(\alpha)$} if the following are satisfied.
\begin{enumerate}
\item[$(A)$] \label{Part-A-D(a)} For every nonempty subset $S$ of\, $[n]$ and every positive integer $\ell$,
\[ \big|\{v\in [n]\colon \dist_G(v,S)\leqslant \ell\}\big|
\geqslant \min\Big\{\frac{3n}{4},\alpha (d-1)^\ell\,|S|\Big\}. \]
\item[$(B)$] \label{Part-B-D(a)} We have $\lambda_2(G) \leqslant 2.1 \sqrt{d-1}$.
\end{enumerate}
\end{definition}

Part (\hyperref[Part-B-D(a)]{$B$}) of property $\mathcal{D}(\alpha)$ is a standard spectral property related to edge expansion. On the other hand, note that part (\hyperref[Part-A-D(a)]{$A$}) quantifies vertex expansion; it was recently introduced in \cite[Definition 1.5, part ($A$)]{ADTT24}.

The second property concerns the ``range" graph $H$.
\begin{definition}[Property $\mathcal{R}(\varepsilon)$] \label{def-R(e)}
Let $m\geqslant \Delta\geqslant 3$ be integers, and let $H$ be a $\Delta$-regular graph~on~$[m]$. Given $0<\varepsilon \leqslant 10^{-4}$, we say that $H$ satisfies \emph{property $\mathcal{R}(\varepsilon)$} if the following are satisfied.
\begin{enumerate}
\item[$(A)$] \label{Part-A-R(e)} For every nonempty subset $S$ of\, $[m]$ with $|S|\leqslant m^{1-\varepsilon}$,
\[ c_{L_1}\big( (S, \dist_H|_{S\times S}) \big) \leqslant \frac{216}{\varepsilon}. \]
\item[$(B)$] \label{Part-B-R(e)} We have that $H$ is connected and $\diam(H) \leqslant 3 \log_{\Delta-1}m$.
\end{enumerate}
\end{definition}

Notice that part (\hyperref[Part-B-R(e)]{$B$}) of property $\mathcal{R}(\varepsilon)$---just like part (\hyperref[Part-B-D(a)]{$B$}) of property $\mathcal{D}(\alpha)$---is also a standard property of random regular graphs. Part (\hyperref[Part-A-R(e)]{$A$}) is more subtle; it was introduced by Mendel and Naor \cite[Section 6]{MN15}, though closely related versions have been studied earlier (see, e.g., \cite{ALNRRV12}).

It is crucial for our analysis that both properties hold with high probability for random graphs.

\begin{proposition}[Properties $\mathcal{D}(\alpha)$ and $\mathcal{R}(\varepsilon)$ are typical] \label{typical}
The following hold.
\begin{enumerate}
\item[(i)] Let $d\geqslant 3$ be an integer, and set
\begin{equation} \label{e1}
\alpha=\alpha(d):= exp\big(-10^{11}\, \log^2 d\big).
\end{equation}
There exists a constant $\tau_1=\tau_1(d)>0$ such that
\begin{equation} \label{e2}
\mathbb{P}_{G(n,d)}\big[G \text{ satisfies property } \mathcal{D}(\alpha)\big] \geqslant 1- O_d\Big(\frac{1}{n^{\tau_1}}\Big).
\end{equation}
\item[(ii)] Let $\Delta\geqslant 3$ be an integer, and let\, $0<\varepsilon \leqslant 10^{-4}$. There exists a constant $\tau_2=\tau_2(\Delta,\varepsilon)>0$ such that
\begin{equation} \label{e3}
\mathbb{P}_{G(m,\Delta)}\big[H \text{ satisfies property } \mathcal{R}(\varepsilon)\big] \geqslant
1- O_{\Delta,\varepsilon}\Big(\frac{1}{m^{\tau_2}}\Big).
\end{equation}
\end{enumerate}
\end{proposition}

\begin{proof}
First we argue for property $\mathcal{D}(\alpha)$. The fact that a uniformly random $d$-regular graph satisfies part (\hyperref[Part-A-D(a)]{$A$}) of property $\mathcal{D}(\alpha)$ with the desired probability follows from \cite[Proposition 1.9]{ADTT24}. Part (\hyperref[Part-B-D(a)]{$B$}) is a direct consequence of Friedman's second eigenvalue theorem \cite{Fr08}.

Towards property $\mathcal{R}(\varepsilon)$, it is a classical observation---see, e.g., \cite[Proposition 1.9]{ADTT24}$\text{---that}$ part (\hyperref[Part-B-R(e)]{$B$}) of property $\mathcal{R}(\varepsilon)$ is satisfied with the desired probability. So, it remains to verify that a uniformly random $\Delta$-regular graph satisfies part (\hyperref[Part-A-R(e)]{$A$}) of property $\mathcal{R}(\varepsilon)$. Fix $0<\varepsilon \leqslant 10^{-4}$, and set
\[\varepsilon_0 := \frac{391}{400}\varepsilon, \ \ \  \delta := \frac{7}{\varepsilon_0 \log_{\Delta}m}
\ \ \ \text{ and } \ \ \ r:=\Big\lceil \frac{14}{625} \varepsilon \log_{\Delta-1}m \Big\rceil.\]
Select a positive integer $m_0=m_0(\Delta,\varepsilon)$ such that for every integer $m\geqslant m_0$, we have $\delta<\frac{1}{3}$ and
\[  m^{1-\varepsilon} \left(3(\Delta-1)m^{\frac{14}{625}\varepsilon}+3\log_{\Delta-1}m+1\right) \leqslant m^{1-\varepsilon_0}; \]
such an $m_0$ exists since $14/625+391/400<1$. We also set
\begin{align*}
\mathcal{H}_m := \Big\{ H\in G(m,\Delta)\colon & H \text{ is connected with }\mathrm{diam}(H)\leqslant3\log_{\Delta-1} m \text{ and} \\
& |E_{H[U]}|\leqslant (1+\delta)|U| \text{ for every } \emptyset\neq U\subseteq [m] \text{ with } |U|\leqslant m^{1-\varepsilon_0}\Big\},
\end{align*}
where $E_{H[S]}$ denotes the edge-set of the induced on $S$ subgraph of $H$. By part (i) of the proposition, $\mathbb{P}_{G(m,\Delta)}[H \text{ is connected and } \diam(H)\leqslant 3\log_{\Delta-1}m]\geqslant 1-O_{\Delta}(m^{-\tau})$ for some $\tau=\tau(\Delta)>0$.
Hence, by \cite[Lemma 7.3]{MN15} and a union bound, we see that $\mathbb{P}_{G(m,\Delta)}\big[\mathcal{H}_m\big] \geqslant 1-O_{\Delta,\varepsilon}(m^{-\tau_2})$ for some $\tau_2=\tau_2(\Delta,\varepsilon)>0$. Thus, it is enough to show that, for every integer $m\geqslant m_0$, every graph in $\mathcal{H}_m$ satisfies part (\hyperref[Part-A-R(e)]{$A$}) of property $\mathcal{R}(\varepsilon)$.
To this end, we will need the following consequence of the work of Mendel--Naor \cite{MN15}.

\begin{claim}[Mendel--Naor] \label{cl3.4*}
Let $\delta\in(0,1/3)$, and let $G$ be a connected graph such that, for any subset $S$ of\, $V_G$, we have $|E_{G[S]}|\leqslant(1+\delta)|S|$. Also let $M\geqslant 2$. Then there exists a map $T_M\colon V_G \to L_1$ such that, for every $x,y\in V_G$,
\begin{equation}\label{e3.5*}
\min\{\mathrm{dist}_G(x,y),M\}\leqslant \|T_M(x) - T_M(y)\|_{L_1} \leqslant \frac{e}{e-1}
\min\Big\{M, \max\Big\{2, 1 + \frac{4\delta M}{1+\delta}\Big\}\mathrm{dist}_G(x,y)\Big\}.
\end{equation}
\end{claim}

\begin{proof}[Proof of Claim \ref{cl3.4*}]

Say a cycle in $G$ is short if its length is less than $\delta/3$. Let $\Gamma$ and $E_\Gamma$ respectively denote the vertices and edges of $G$ that belong to a short cycle. As noted in \cite[Lemma 6.10]{MN15}, every $e\in E_\Gamma$ is contained in a unique short cycle. Denote by $T_{\mathrm{good}}(G)$ the set of all spanning subtrees $T$ of $G$ which satisfy $|E_T\cap E_G(C)| = |C| - 1$ for every short cycle $C$. It is proven in \cite[Lemma 6.14]{MN15} that there exists a probability measure $\mu$ on $T_{\mathrm{good}}(G)$ with the following two properties. First, for every $e = \{x,y\} \in E_\Gamma$,
\[ \mathbb{E}_{\mu}\big[ \min\{M,  \mathrm{dist}_T(x,y) \} \big] \leqslant 1 + \frac{1}{|C_e|} \big(|C_e|-1\big)\leqslant 2. \]
Second, for every $e = \{x,y\} \in E_G\setminus E_\Gamma$,
\[ \mathbb{E}_{\mu}\big[ \min\{M,  \mathrm{dist}_T(x,y) \} \big] \leqslant  1 + \frac{4\delta}{1+\delta}\,M. \]
Therefore, for every $x,y\in V_G$, summing the above inequality along a shortest $G$-path from $x$ to $y$,
\begin{equation}\label{e3.6*}
\min\{M,  \mathrm{dist}_G(x,y) \} \leqslant \mathbb{E}_{\mu}\big[ \min\{M,  \mathrm{dist}_T(x,y) \} \big]
\leqslant  A_{\delta,M}\, \mathrm{dist}_G(x,y).
\end{equation}
The lower bound here follows by noticing that $\mathrm{dist}_G(x,y) \leqslant \mathrm{dist}_T(x,y)$. Next, as in \cite{MN15}, for every $T\in T_{\mathrm{good}}(G)$, we select an isometry $F_T\colon (V_G,\mathrm{dist}_T) \to L_1$. By \cite[Lemma 5.4]{MN15}, there exists a map $\Phi_M\colon L_1\to L_1$ such that, for every $f,g\in L_1$,
\begin{equation} \label{e3.7*}
\min\big\{ M, \|f-g\|_{L_1} \big\} \leqslant \big\|\Phi(f) - \Phi(g)\big\|_{L_1} \leqslant \frac{e}{e-1} \min\big\{ M, \|f-g\|_{L_1}\big\}.
\end{equation}
Define $T_M\colon V_G \to L_1^{T_{\mathrm{good}}(G)}$ by setting $T_M(x) := \big(\Phi(F_T(x))\big)_{T\in T_{\mathrm{good}}(G)}$. Then, for every $x,y\in V_G$,
\begin{align*}
\min\{\mathrm{dist}_G(x,y),M\} & \stackrel{\eqref{e3.6*}}{\leqslant} \mathbb{E}_{\mu}\big[ \min\{M,  \mathrm{dist}_T(x,y) \} \big]
\stackrel{\eqref{e3.7*}}{\leqslant}  \mathbb{E}_{\mu}\big[ \| \Phi(F_T(x)) -\Phi(F_T(y)) \|_{L_1} \big] \\
& \ = \|T_M(x) - T_M(y)\|_{L_1(\mu, L_1)}
\stackrel{\eqref{e3.7*}}{\leqslant}  \frac{e}{e-1} \mathbb{E}_{\mu}\big[ \min\{M,  \mathrm{dist}_T(x,y) \} \big] \\
& \ \leqslant \frac{e}{e-1}
\min\Big\{M, A_{\delta,M}\Big\}\,\mathrm{dist}_G(x,y)\Big\},
\end{align*}
where the last inequality follows by \eqref{e3.6*} and the fact that $\mathbb{E}_{\mu}\big[ \min\{M,  \mathrm{dist}_T(x,y) \}\big]\leqslant M$. The proof of Claim \ref{cl3.4*} is completed.
\end{proof}

Fix $m\geqslant m_0$, let $H\in\mathcal{H}_m$, and let $\emptyset\neq S\subseteq[m]$ satisfy $|S|\leqslant m^{1 -\varepsilon}$. Set $D:=\diam(H)$, select $o\in S$, and, for every $s\in S$, let $P_s$ be a shortest path in $H$ connecting $s$ to $o$ and let $V(P_s)$ denote the vertices along this path. Define
\[ U:=B_H(S,r)\cup\bigcup_{s\in S}V(P_s). \]
Since $|B_H(s,r)| \leqslant \frac{\Delta}{\Delta-2}(\Delta-1)^r \leqslant 3(\Delta-1)m^{\frac{14}{625}\varepsilon}$, we obtain
\begin{align*}
    |U| &\leqslant |S| \left( 3(\Delta-1)m^{\frac{14}{625}\varepsilon} +D+1 \right) \leqslant m^{1-\varepsilon} \left( 3(\Delta-1)m^{\frac{14}{625}\varepsilon} +3\log_{\Delta-1}m+1 \right) \leqslant m^{1-\varepsilon_0}.
\end{align*}

Define the induced subgraph $G:=H[U]$. The paths $P_s$ make $G$ connected, and the definition of $\mathcal{H}_m$ shows that $G$ satisfies the assumption of Claim \ref{cl3.4*}. We apply the claim with $M=D$ and abbreviate $A := A_{\delta,D}$. For every $x,y\in S$, since $G$ is a subgraph of $H$ and $\dist_H(x,y)\leqslant D$, we have
\[  \dist_H(x,y) \leqslant \min \{\dist_G(x,y),D\} \leqslant \|T_D(x)-T_D(y)\|_{L_1}. \]
Notice that if $\dist_H(x,y)\leqslant r$, then a shortest $H$-path from $x$ to $y$ is contained in $B_H(S,r)$, and hence $\dist_G(x,y)=\dist_H(x,y)$; therefore, by Claim \ref{cl3.4*},
\[ \|T_D(x)-T_D(y)\|_{L_1} \leqslant \frac{e}{e-1}A\,\dist_H(x,y). \]
If, on the other hand, $\dist_H(x,y)>r$, then
\[ \|T_D(x)-T_D(y)\|_{L_1} \leqslant \frac{e}{e-1}D \leqslant \frac{e}{e-1}\frac{D}{r}\,\dist_H(x,y). \]
Combining the previous three displays we obtain
\[ c_{L_1}\bigl((S,\dist_H|_{S\times S})\bigr) \leqslant \frac{e}{e-1}\max\left\{A,\frac{D}{r}\right\}. \]
It remains only to estimate the two terms. Since $D\leqslant3\log_{\Delta-1}m$, we have
\begin{align*}
\frac{e}{e-1}\,A &\leqslant \frac{4e}{e-1}(1+\delta D) \leqslant\frac{4e}{e-1}\left(1+\frac{21}{\varepsilon_0}\frac{\ln3}{\ln2}\right) \leqslant \frac{216}{\varepsilon},
\end{align*}
and, by the definition of $r$,
\[ \frac{e}{e-1}\frac{D}{r}\leqslant\frac{1875e}{14(e-1)\varepsilon}\leqslant \frac{216}{\varepsilon}. \]
This yields the desired bound $c_{L_1}\bigl((S,\dist_H|_{S\times S})\bigr)\leqslant\frac{216}{\varepsilon}$, and the proof of Proposition \ref{typical} is thus completed.
\end{proof}

\subsection{Decomposition, and main propositions}

Let $G\in G(n,d)$ and $H\in G(m,\Delta)$ be two regular graphs, and let $f\colon [n]\to [m]$ be a function from the vertex set of $G$ into the vertex set of~$H$. The proofs of Theorems \ref{thm-kleinberg} and \ref{mn-random} depend on the behavior of the function $f$, especially in regard to whether $f$ is injective, constant, or somewhere in between. The following definition quantifies this behavior.

\begin{definition}[Decomposition] \label{def-deco}
For every $0<\varepsilon \leqslant 10^{-4}$ and every $f\colon [n]\to [m]$, where $n,m$ are positive integers, we set
\begin{align} \label{e5}
M_0(f,\varepsilon) & := \Big\{ v \in [n]\colon \big|f^{-1}\big(f(v)\big)\big| \leqslant
\Big(\frac{n}{m}\Big) m^{2\varepsilon} \Big\}, \\
\label{e6} M_1(f,\varepsilon) & := \Big\{ v \in [n]\colon \Big(\frac{n}{m}\Big) m^{2\varepsilon} <
\big|f^{-1}\big(f(v)\big)\big| \leqslant \Big(\frac{n}{m}\Big) m^{4\varepsilon} \Big\}, \\
\label{e7} M_2(f,\varepsilon) & := \Big\{ v \in [n] \colon \big|f^{-1}\big(f(v)\big)\big| >
\Big(\frac{n}{m}\Big) m^{4\varepsilon} \Big\}.
\end{align}
\end{definition}
The reader can roughly think of the set $M_0(f,\varepsilon)$ as the part of the domain of $f$ on which the function $f$ is essentially injective, while $M_1(f,\varepsilon)$ and $M_2(f,\varepsilon)$ as the parts of the domain of $f$ on which the function $f$ in nearly constant (in the sense that the image of $M_1(f,\varepsilon)$ and $M_2(f,\varepsilon)$ is a ``thin" subset of $[m]$).

The following proposition treats the case when the set $M_2(f,\varepsilon)$ is small.

\begin{proposition}[Validity of Poincar\'{e} inequality if $M_2(f,\varepsilon)$ is small] \label{small-M2}
Let $d,\Delta\geqslant 3$ be integers, let\, $0<\varepsilon\leqslant 10^{-4}$, and set
\begin{equation} \label{e8}
\Gamma_1=\Gamma_1(\varepsilon) := \frac{18}{\varepsilon}.
\end{equation}
There exists a positive integer $m_0=m_0(d)$ with the following property. Let $m\geqslant m_0$ be an integer, and let $H\in G(m,\Delta)$ be connected with $\diam(H)\leqslant 3 \log_{\Delta-1}m$. Then, for any integer $n\geqslant 4d$,
\begin{align} \label{e9}
\mathbb{P}_{G(n,d)} \bigg[G\colon \frac{1}{n^2}  \sum_{v,u\in [n]} \dist_H&\big(f(v),f(u)\big) \leqslant
\Gamma_1 \, \frac{1}{|E_G|} \sum_{\{v,u\}\in E_G}  \dist_H\big(f(v),f(u)\big), \\
& \text{for every } f\colon [n]\to [m] \text{ with } |M_2(f,\varepsilon)|\leqslant \frac{n}{8}\bigg]
\geqslant 1-O_{d}\Big( \frac{1}{m^{0.05\, n}}\Big). \nonumber
\end{align}
\end{proposition}

The proof of Proposition \ref{small-M2} is given in Section \ref{sec4}, and it relies on strong concentration bounds that are available if $|M_2(f,\varepsilon)|\leqslant \frac{n}{8}$.

The next proposition complements Proposition \ref{small-M2} and treats the case when the set $M_2(f,\varepsilon)$~is large.

\begin{proposition}[Validity of Poincar\'{e} inequality if $M_2(f,\varepsilon)$ is large] \label{large-M2}
Let $d,\Delta\geqslant 3$ be integers, let~$0<\alpha\leqslant 1$, let\,  $0<\varepsilon \leqslant 10^{-4}$, and set
\begin{equation} \label{e10}
\Gamma_2=\Gamma_2(d,\alpha,\varepsilon) := \frac{1}{\alpha\cdot \varepsilon^2} \, \exp\big( 3632\cdot \log^2 d\big).
\end{equation}
There exists a positive integer $m_1=m_1(d,\alpha,\varepsilon)$ with the following property. Let $n\geqslant d$ and $m\geqslant m_1$ be integers, and let\, $G\in G(n,d)$ and $H\in G(m,\Delta)$. Assume that $G$ satisfies property~$\mathcal{D}(\alpha)$, and that $H$ satisfies property $\mathcal{R}(\varepsilon)$. Then, for any $f\colon [n]\to [m]$ with $|M_2(f,\varepsilon)|\geqslant \frac{n}{8}$,
where $M_2(f,\varepsilon)$ is as in \eqref{e7}, we have
\begin{equation} \label{e11}
\frac{1}{n^2} \sum_{v,u\in [n]} \dist_H\big(f(v),f(u)\big) \leqslant
\Gamma_2 \, \frac{1}{|E_G|} \sum_{\{v,u\}\in E_G}  \dist_H\big(f(v),f(u)\big).
\end{equation}
\end{proposition}

The proof of Proposition \ref{large-M2} is given in Section \ref{sec5}, and it is significantly more demanding. It relies on properties $\mathcal{D}(\alpha)$ and $\mathcal{R}(\varepsilon)$, a dyadic decomposition of the set $M_0(f,\varepsilon)$, and a randomized ``compression" argument.

\section{Proof of Proposition \ref*{small-M2}} \label{sec4}

Fix $d,\Delta,\varepsilon,m,H$ and $n$; clearly, we may assume that $dn$ is even. Let $\mathcal{F}$ denote the set of all maps $f\colon [n] \to [m]$ with $|M_2(f,\varepsilon)|\leqslant \frac{n}{8}$. For every $f\in\mathcal{F}$, set
\[ \mathcal{G}_f := \bigg\{ G\in G(n,d) \colon \frac{1}{n^2}  \sum_{v,u\in [n]} \dist_H\big(f(v),f(u)\big)
\leqslant \Gamma_1 \, \frac{1}{|E_G|} \sum_{\{v,u\}\in E_G}  \dist_H\big(f(v),f(u)\big) \bigg\}. \]
Since $|\mathcal{F}|\leqslant m^n$, by a union bound, it is enough to show that for every $f\in \mathcal{F}$,
\begin{equation} \label{e4.1}
\mathbb{P}_{G(n,d)}\big[ \mathcal{G}_f\big] \geqslant 1-O_{d}\Big( \frac{1}{m^{1.05 n}}\Big).
\end{equation}
So, fix $f\in \mathcal{F}$, and set $R:= [n]\setminus M_2(f,\varepsilon)$ and
\[  \mathcal{A} := \bigg\{ \{v,u\} \in \binom{R}{2} \colon \mathrm{dist}_H \big(f(v), f(u)\big)
\leqslant 4\varepsilon \log_{\Delta-1}m \bigg\}. \]
Notice that by the $\Delta$-regularity of $H$, the definition of $M_2(f,\varepsilon)$ and the fact that $\Delta\geqslant 3$, we have
\begin{equation} \label{e4.2}
|\mathcal{A}| \leqslant \frac{3}{2}\, |R|\, n\, m^{8\varepsilon-1}.
\end{equation}
Also observe that, since $|M_2(f,\varepsilon)|\leqslant\frac{n}{8}$, for every $G\in G(n,d)$,
\begin{equation} \label{e4.3}
|E_{G[R]}|\geqslant \frac{dn}{2} - d|M_2(f,\varepsilon)|\geqslant \frac{3dn}{8},
\end{equation}
where $E_{G[R]}$ denotes the edge-set of the induced on $R$ subgraph of $G$. For every $k\in [\frac{dn}{2}]$ with $k\geqslant\frac{3dn}{8}$, set
\[ \mathcal{E}_k := \big\{ G \in G(n)\colon |E_{G[R]}|=k \big\}, \]
where $G(n)$ denotes the set of all graphs on $[n]$. Since $\diam(H)\leqslant3\log_{\Delta-1}m$, it holds, for every $k\in [\frac{dn}{2}]$ with $k\geqslant\frac{3dn}{8}$ and every $G\in(G(n,d)\setminus\mathcal{G}_f)\cap\mathcal{E}_k$, that
\begin{equation*}
\frac{1}{|E_{G[R]}|}\sum_{\{v,u\}\in E_{G[R]}} \dist_H\big(f(v),f(u)\big)
\leqslant \frac{4}{3|E_G|} \sum_{\{v,u\}\in E_G} \dist_H\big(f(v),f(u)\big) \leqslant \frac{4}{\Gamma_1}\log_{\Delta-1}m.
\end{equation*}
Therefore, by Markov's inequality, the definition of $\mathcal{A}$, and the choice of $\Gamma_1$ in \eqref{e8}, we see that $\big|E_{G[R]}\cap\big(\binom{R}{2}\setminus\mathcal{A}\big)\big| \leqslant k/18$. Set $p:=d/(n-1)$, and let $\mathbb{P}_{G(n,p)}$ denote the law of the classical (independent) Erd\H{o}s--R\'enyi random graph. Then the previous observations imply, after conditioning on the number of edges contained in $\binom{[n]}{2}\setminus\binom{R}{2}$, that for every $k\in [\frac{dn}{2}]$ with $k\geqslant\frac{3dn}{8}$,
\begin{equation}\label{e4.4}
\mathbb{P}_{G(n,p)}\big[ G(n,d)\setminus\mathcal{G}_f\, \big| \, \mathcal{E}_k\big]
\leqslant \mathbb{P}_{G(n,p)}\big[ |\mathcal{A} \cap E_{\mathbf{G}[R]}| \geqslant 17k/18\, \big| \, \mathcal{E}_k\big]
\leqslant \sum_{j= \lceil 17k/18 \rceil}^{k} \frac{\binom{|\mathcal{A}|}{j}
\binom{|\binom{R}{2}\setminus\mathcal{A}|}{k-j}}{\binom{{\binom{|R|}{2}}}{k}},
\end{equation}
where $\mathbf{G}\sim \mathbb{P}_{G(n,p)}$. Since $n\geqslant 4d$, $\varepsilon\leqslant 10^{-4}$, $d\geqslant3$ and $|R|\geqslant \frac{7n}{8}$, by \eqref{e4.2}, we see that, for every $m\geqslant 28^{1000}$, the right-hand side of \eqref{e4.4} is upper bounded by
\begin{equation} \label{e4.5}
\frac{k\binom{|\mathcal{A}|}{\lceil 17k/18 \rceil}
\binom{|\binom{R}{2}\setminus\mathcal{A}|}{k-\lceil 17k/18\rceil}}{18\binom{{\binom{|R|}{2}}}{k}}
\leqslant \Bigg(\frac{4|\mathcal{A}|}{\binom{|R|}{2}-k}\Bigg)^{17k/18}
\leqslant \big(28 m^{8\varepsilon-1}\big)^{\frac{3\cdot17}{8\cdot18}dn}
\leqslant m^{-1.059\cdot \frac{d}{3}n}.
\end{equation}
To complete the proof, we will use the following asymptotic enumeration of regular graphs \cite{MW91}.

\begin{fact} \label{fact-asymp-enum}
If $n \geqslant d\geqslant 3$ are integers with $dn$ even, then the cardinality of\, $G(n,d)$ is of~order
\[ \Theta_d\bigg(\frac{(d n)!}{(d n/2)! 2^{d n/2} (d!)^n}\bigg); \]
in particular, there exists a constant $C_1=C_1(d)>0$ such that
\[ \mathbb{P}_{G(n,p)}\big[ G(n,d)\big] \geqslant \exp(-C_1 n), \]
where $p=d/(n-1)$ and\, $\mathbb{P}_{G(n,p)}$ denotes the law of the Erd\H{o}s--R\'enyi random graph.
\end{fact}

Since $d\geqslant3$, by \eqref{e4.3}--\eqref{e4.5} and Fact \ref{fact-asymp-enum}, there exists a positive integer $m_0=m_0(d)\geqslant 28^{1000}$ such that for every $n\geqslant 4d$ and every $m\geqslant m_0$,
\begin{equation}\label{e4.6}
\mathbb{P}_{G(n,d)}\big[G(n,d)\setminus\mathcal{G}_f \big] =
\frac{\mathbb{P}_{G(n,p)}\big[G(n,d)\setminus\mathcal{G}_f \big]}{\mathbb{P}_{G(n,p)}\big[G(n,d) \big]}
\leqslant m^{-1.05\cdot \frac{d}{3}n} \leqslant m^{-1.05\, n}.
\end{equation}
The proof is completed after observing that \eqref{e4.1} follows from \eqref{e4.6}.

\section{Proof of Proposition \ref*{large-M2}} \label{sec5}

\subsection{Preliminary tools}

We start by isolating some preliminary results that are needed for the proof of Proposition \ref{large-M2}. The first one is a consequence of Fact \ref{fact-L1}.

\begin{fact} \label{fact5.1}
Let $m\geqslant\Delta\geqslant3$ be integers, let $0<\varepsilon\leqslant10^{-4}$, and let $H\in G(m,\Delta)$ satisfy property~$\mathcal{R}(\varepsilon)$. Then, for every pair of integers $n\geqslant d\geqslant3$, every $G\in G(n,d)$ with $\lambda_2(G)\neq d$, and every function $f\colon [n]\to[m]$ with $|\mathrm{Im}(f)|\leqslant m^{1-\varepsilon}$, we have
\begin{equation} \label{e5.1}
\frac{1}{n^2} \sum_{v,u\in [n]} \mathrm{dist}_H\big(f(v), f(u)\big) \leqslant
\Big(\frac{216}{\varepsilon} \cdot\frac{d}{2(d-\lambda_2(G))}\Big) \,
\frac{1}{|E_G|} \sum_{\{v,u\}\in E_G} \mathrm{dist}_H\big(f(v), f(u)\big).
\end{equation}
In particular, if\, $G$ satisfies property $\mathcal{D}(\alpha)$ for some $0<\alpha\leqslant1$, then for any $f\colon [n]\to[m]$ with $|\mathrm{Im}(f)|\leqslant m^{1-\varepsilon}$,
\begin{equation}\label{e5.2}
\frac{1}{n^2} \sum_{v,u\in [n]} \mathrm{dist}_H\big(f(v), f(u)\big) \leqslant
\frac{10746}{\varepsilon} \cdot \frac{1}{|E_G|} \sum_{\{v,u\}\in E_G} \mathrm{dist}_H\big(f(v), f(u)\big).
\end{equation}
\end{fact}

\begin{proof}
Inequality \eqref{e5.1} follows from  Fact \ref{fact-L1} and part (\hyperref[Part-A-R(e)]{$A$}) of property $\mathcal{R}(\varepsilon)$, while \eqref{e5.2} follows from \eqref{e5.1} and part (\hyperref[Part-B-D(a)]{$B$}) of property $\mathcal{D}(\alpha)$.
\end{proof}

The next lemma is an expansion result that is available for all regular graphs satisfying part~(\hyperref[Part-B-D(a)]{$B$}) of property $\mathcal{D}(\alpha)$.

\begin{lemma} \label{lemma5.2}
Let $n\geqslant d\geqslant3$ be integers,\, let $0<\xi<1$, and let\, $G\in G(n,d)$ with $\lambda_2(G)\leqslant2.1\sqrt{d-1}$. Then, for every $A\subseteq [n]$ with $|A|\leqslant (1-\xi)n$, we have
\begin{equation} \label{e5.3}
\big|\{e\in E_G\colon e\cap A\neq\emptyset \text{ and } e\setminus A\neq\emptyset\}\big| \geqslant
\frac{\xi}{1-\xi} \cdot \frac{d}{200}\, |A|
\end{equation}
and
\begin{equation} \label{e5.4}
\big|\{v\in [n]\colon \mathrm{dist}_G(v,A)=1\}\big| \geqslant \frac{\xi}{1-\xi} \cdot \frac{1}{200}\, |A|.
\end{equation}
\end{lemma}

\begin{proof}
Fix $A\subseteq [n]$ with $|A|\leqslant (1-\xi)n$. Then, $\min\big\{|A|,|A^\complement|\big\}\geqslant \frac{\xi}{1-\xi}|A|$ and so, by~\eqref{e-cheeger-e1} and our assumption that $\lambda_2(G)\leqslant2.1\sqrt{d-1}$, we see that
\[ \big| \{e\in E_G\colon e\cap A\neq\emptyset \text{ and } e\setminus A\neq\emptyset\}\big| \geqslant
h(G) \cdot \min\big\{|A|,|A^\complement|\big\}  \geqslant \frac{\xi}{1-\xi} \cdot \frac{d}{200} |A|, \]
where we have used the fact that $d-2.1\sqrt{d-1}\geqslant 0.01 d$ for every integer $d\geqslant 3$; in particular, \eqref{e5.3} is satisfied. Inequality \eqref{e5.4} follows from \eqref{e5.3} and the $d$-regularity of $G$.
\end{proof}

The last result that we will need is the following mild strengthening of part (\hyperref[Part-A-D(a)]{$A$}) of property~$\mathcal{D}(\alpha)$.

\begin{lemma} \label{lemma5.3}
Let $n\geqslant d\geqslant3$ be integers, let\, $0<\alpha\leqslant1$, and let $G\in G(n,d)$ that satisfies property~$\mathcal{D}(\alpha)$. Set
\begin{equation}\label{e5.5}
\ell_*=\ell_*(d):=\Big\lfloor\log_{1+\frac{1}{3\cdot 10^3}}\Big(\frac{5(d-1)}{4}\Big)\Big\rfloor \ \ \ \text{ and } \ \ \
\tilde{\alpha}=\tilde{\alpha}(\alpha,d):= \frac{5\alpha}{4} \Big(\frac{1}{d-1}\Big)^{\ell_*}.
\end{equation}
Then, for every nonempty subset $S$ of\, $[n]$ and every positive integer $\ell$,
\begin{equation}\label{e5.6}
\big|\{v\in [n]\colon \dist_G(v,S)\leqslant \ell\}\big|
\geqslant \min\Big\{\frac{15n}{16},\tilde{\alpha} (d-1)^\ell\,|S|\Big\}.
\end{equation}
\end{lemma}

\begin{proof}
First, observe that by repeatedly applying Lemma \ref{lemma5.2} for $\xi = 1/16$ and invoking the choice of $\tilde{\alpha}$ in \eqref{e5.5}, we have for every nonempty $A\subseteq [n]$ and every positive integer $\ell$,
\begin{equation}\label{e5.7}
|B_G(A,\ell)| \geqslant \min\bigg\{ \frac{15n}{16}, \Big(1 +\frac{1}{3\cdot 10^3} \Big)^\ell |A| \bigg\}.
\end{equation}
(Recall that $B_G(A,\ell)$ denotes the set $\big\{v\in [n]\colon \dist_G(v,A)\leqslant \ell\big\}$.) Set
\[ \alpha_1=\alpha_1(d):= \Big(\frac{1+\frac{1}{3\cdot10^3}}{d-1}\Big)^{\ell_*}. \]

\begin{claim} \label{cl5.4}
For every $A\subseteq [n]$ with $|A|\geqslant \frac{3n}{4(d-1)}$ and every positive integer $\ell$, we have
\begin{equation}\label{e5.8}
|B_G(A, \ell)| \geqslant \min\Big\{ \frac{15n}{16}, \alpha_1(d-1)^\ell |A| \Big\}.
\end{equation}
\end{claim}

\begin{proof}[Proof of Claim \ref{cl5.4}]
Assume, first, that $\big(1 +\frac{1}{3\cdot 10^3} \big)^\ell |A| \leqslant \frac{15n}{16}$; then $\ell\leqslant\ell_*$ and consequently, by~\eqref{e5.7},
\[ |B_G(A,\ell)|\geqslant \Big(1 +\frac{1}{3\cdot 10^3} \Big)^\ell |A| =
\Big( \frac{1 +\frac{1}{3\cdot 10^3}}{d-1} \Big)^\ell (d-1)^\ell |A| \geqslant \alpha_1 (d-1)^\ell |A|, \]
as desired. On the other hand, if $\big(1 +\frac{1}{3\cdot 10^3} \big)^\ell |A| > \frac{15n}{16}$, then, by \eqref{e5.7}, we have $|B_G(A,\ell)|\geqslant \frac{15n}{16}$, and the claim follows.
\end{proof}

We are ready to prove \eqref{e5.6}. Towards this end, fix a nonempty subset $S$ of $[n]$ and a positive integer $\ell$. Notice that if $|S|\geqslant \frac{3n}{4(d-1)}$, then \eqref{e5.6} follows from Claim \ref{cl5.4} and the fact that $\tilde{\alpha}\leqslant\alpha_1$. So, assume that $|S|<\frac{3n}{4(d-1)}<\frac{3n}{4}$. Set
\[ \ell_1:= \min\Big\{\ell\geqslant 1\colon \alpha(d-1)^\ell|S|\geqslant \frac{3n}{4}\Big\}-1, \]
and observe that
\begin{equation}\label{e5.9}
\frac{3}{4(d-1)} n \leqslant \alpha(d-1)^{\ell_1}|S|< \frac{3}{4}n.
\end{equation}
If $\ell\leqslant\ell_1$, then \eqref{e5.6} follows by part (\hyperref[Part-A-D(a)]{$A$}) of property~$\mathcal{D}(\alpha)$, inequality \eqref{e5.9} and the fact that $\tilde{\alpha}\leqslant\alpha$. Finally, assume that $\ell>\ell_1$, and set $A:=B_G(S,\ell_1)$. By \eqref{e5.9} and (\hyperref[Part-A-D(a)]{$A$}) of property~$\mathcal{D}(\alpha)$, we have
\[ |A|\geqslant\frac{3n}{4(d-1)}> \frac{\alpha}{d-1}(d-1)^{\ell_1}|S|. \]
Therefore, by Claim \ref{cl5.4}, we conclude that
\[ |B_G(S,\ell)| = |B_G(A,\ell-\ell_1)| \geqslant
\min\Big\{ \frac{15n}{16}, \alpha_1(d-1)^{\ell-\ell_1} |A| \Big\} \geqslant
\min\Big\{ \frac{15n}{16}, \tilde{\alpha}(d-1)^{\ell} |A| \Big\}. \qedhere \]
\end{proof}

\subsection{Completion of the proof} \label{subsec5.2}

For the rest of this section, fix $d,\Delta, \alpha,\varepsilon, n, m, G$ and $H$. We also fix $f\colon [n]\to[m]$ such that $|M_2(f,\varepsilon)|\geqslant\frac{n}{8}$. Notice that if $M_0(f,\varepsilon)=\emptyset$, then, by Definition \ref{def-deco}, we obtain that $|\mathrm{Im}(f)|\leqslant m^{1-2\varepsilon}\leqslant m^{1-\varepsilon}$. Since $10746/\varepsilon\leqslant \Gamma_2$, by Fact \ref{fact5.1}, we see that \eqref{e11} is satisfied. Thus, we may assume that $M_0(f,\varepsilon)\neq \emptyset$ in what follows.

As previously noted, the set $M_0(f,\varepsilon)$ is the portion of the domain of $f$ on which $f$ is essentially injective. Note, however, that the cardinality of the set $f^{-1}\big(f(v)\big)$ may fluctuate as $v$ varies over $M_0(f,\varepsilon)$. This motivates the following dyadic decomposition of $M_0(f,\varepsilon)$.

\begin{definition}[Dyadic decomposition] \label{def-dyadic}
For every $r\in\mathbb{Z}$, set
\begin{equation}\label{e5.10}
M_{0,r} := \Big\{ v\in M_0(f,\varepsilon)\colon 2^{r-1} \leqslant \big|f^{-1}\big(f(v)\big)\big| <2^r\Big\};
\end{equation}
note that $M_{0,r} = \emptyset$ if $r\leqslant 0 $ or if $r> \log_2(m^{2\varepsilon-1}\,n)+1$. We also set
\begin{equation}\label{e5.11}
r^* := \max\big\{r\colon M_{0,r} \neq \emptyset\big\}.
\end{equation}
Finally, define
\begin{equation} \label{e5.12}
M_0':= \bigcup_{\substack{r^* - 2 \log_2 m \leqslant r \leqslant r^* \\ |f(M_{0,r})| \geqslant m^{1-2\varepsilon}}} M_{0,r}
\ \ \ \ \ \text{ and } \ \ \ \ \  M_0'':= \bigcup_{r < r^* - 2 \log_2 m} M_{0,r}.
\end{equation}
Observe that $M_0'\cup M_0''\subseteq M_0(f,\varepsilon)=M_{0,1}\cup \cdots \cup M_{0,r^*}$.
\end{definition}

We also introduce the following definition.

\begin{definition}[Random compression] \label{def:compression}
We define the (deterministic) function $\hat f\colon [n] \to [m]$ by
\begin{equation} \label{eq-hat(f)}
\hat f(v) :=
\begin{cases}
1 & v \in M_0'', \\
f(v) & v\in [n] \setminus M_0''.
\end{cases}
\end{equation}
Notice that $\hat f = f$ if $M_0''=\emptyset$. We also define the \emph{random compression} $\boldsymbol{f}\colon [n]\to[m]$ of $f$ as follows. We select uniformly at random a vertex $\mathbf{v}$ from $M_0'$, and we set
\begin{equation} \label{eq-random(f)}
\boldsymbol{f}(v) :=
\begin{cases}
\hat f(\mathbf{v}) & v \in M_0', \\
\hat f(v) & v\in [n] \setminus M_0'.
\end{cases}
\end{equation}
Again observe that if $M_0'=\emptyset$, then $\boldsymbol{f}$ deterministically equals to $\hat f$.
\end{definition}

In other words, the deterministic function $\hat{f}$ is obtained by compressing the image $f(M_0'')$ of the set $M_0''$ into the singleton $\{1\}$. Subsequently, the random compression $\boldsymbol{f}$ is defined by selecting a vertex $\mathbf{v}\in M_0'$ uniformly at random, and then compressing the image $f(M_0')$ of the set $M_0'$ into the singleton $\{f(\mathbf{v})\}$. Crucially, this construction allows us to leverage the following.

\begin{fact} \label{obs5.7}
There exists a positive integer $m_2=m_2(\varepsilon)$ such that for every integer $m\geqslant m_2$, we have, almost everywhere,
\begin{equation} \label{e5.13}
\frac{1}{n^2} \sum_{v,u\in [n]} \mathrm{dist}_H\big(\boldsymbol{f}(v), \boldsymbol{f}(u)\big) \leqslant
\frac{10746}{\varepsilon} \cdot
\frac{1}{|E_G|} \sum_{\{v,u\}\in E_G} \mathrm{dist}_H\big(\boldsymbol{f}(v), \boldsymbol{f}(u)\big).
\end{equation}
\end{fact}

\begin{proof}
Set $m_2(\varepsilon):= \big\lceil (4+2\log_2 m)^{\frac{1}{\varepsilon}}\big\rceil$. If $m\geqslant m_2(\varepsilon)$, then, since $\Delta\geqslant3$, it holds almost everywhere,
\begin{align*}
\big|\mathrm{Im}(\boldsymbol{f})\big| & \leqslant 2 + \big|f\big(M_0(f,\varepsilon)\setminus (M_0''\cup M_0')\big)\big| +
\big| f\big(M_1(f,\varepsilon)\cup M_2(f,\varepsilon)\big)\big| \\
& \leqslant 2 + (2\log_2 m+1)m^{1-2\varepsilon} + m^{1-2\varepsilon} \leqslant m^{1-\varepsilon}.
\end{align*}
Therefore, inequality \eqref{e5.13} follows from Fact \ref{fact5.1}.
\end{proof}

The following lemma is the last piece needed to complete the proof of Proposition \ref{large-M2}.

\begin{lemma}[Main lemma] \label{l5.8}
Let the notation and assumptions be as in Subsection \ref{subsec5.2}, and set
\begin{equation}\label{e5.14}
\eta=\eta(d,\alpha,\varepsilon):= \frac{\tilde{\alpha}\, \varepsilon}{2^{8}\, 3^2\, 10^3},
\end{equation}
where $\tilde{\alpha}$ is as in \eqref{e5.5}. There exists a positive integer $m_3=m_3(d,\alpha,\varepsilon)$ such that for every integer $m\geqslant m_3$, we have, almost everywhere,
\begin{equation}\label{e5.15}
\sum_{\{v,u\}\in E_G}\mathrm{dist}_H\big( \boldsymbol{f}(v), \boldsymbol{f}(u)\big) \leqslant
\frac{4d}{\eta} \sum_{\{v,u\}\in E_G}\mathrm{dist}_H\big( f(v), f(u) \big)
\end{equation}
and, on expectation,
\begin{equation}\label{e5.16}
\mathbb{E}\bigg[ \sum_{v,u\in [n]} \mathrm{dist}_H\big( \boldsymbol{f}(v), \boldsymbol{f}(u)  \big) \bigg] \geqslant
\frac{1}{16} \sum_{v,u\in [n]} \mathrm{dist}_H\big( f(v), f(u) \big).
\end{equation}
\end{lemma}

Before proving this key technical estimate, let us complete the proof of Proposition \ref{large-M2}. By Fact~\ref{obs5.7} and Lemma \ref{l5.8}, if $m$ is sufficiently large in terms of $d$ and $\varepsilon$,~then
\begin{align*}
\frac{1}{n^2} \sum_{v,u\in [n]} \mathrm{dist}_H\big( f(v), f(u) \big) & \stackrel{\eqref{e5.16}}{\leqslant}
\frac{16}{n^2} \, \mathbb{E}\bigg[ \sum_{v,u\in [n]} \mathrm{dist}_H\big( \boldsymbol{f}(v), \boldsymbol{f}(u)  \big) \bigg] \\
& \stackrel{\eqref{e5.13}}{\leqslant} \frac{16\cdot 10746}{\varepsilon} \cdot \frac{1}{|E_G|}\,
\mathbb{E}\bigg[\sum_{\{v,u\}\in E_G} \mathrm{dist}_H\big(\boldsymbol{f}(v), \boldsymbol{f}(u)\big) \bigg] \\
& \stackrel{\eqref{e5.15}}{\leqslant}  \frac{4^3\cdot 10746\cdot d}{\eta\varepsilon} \cdot \frac{1}{|E_G|}\,
\sum_{\{v,u\}\in E_G}\mathrm{dist}_H\big( f(v), f(u) \big) \\
& \ \leqslant \Gamma_2\, \frac{1}{|E_G|}\, \sum_{\{v,u\}\in E_G}\mathrm{dist}_H\big( f(v), f(u) \big),
\end{align*}
where the last inequality follows from the choice of $\Gamma_2$ in \eqref{e10}, and the choices of $\tilde{\alpha}$ and $\eta$ in \eqref{e5.5} and \eqref{e5.14}, respectively.

\subsection{Proof of Lemma \ref{l5.8}} \label{subsec5.3}

The proof is decomposed into two parts. The first relies crucially on part~(\hyperref[Part-A-D(a)]{$A$}) of property~$\mathcal{D}(\alpha)$.

\begin{claim} \label{claim5.9}
There exists a positive integer $m_4=m_4(d,\alpha,\varepsilon)$ such that for every integer $m\geqslant m_4$,
\begin{equation} \label{e5.17}
\sum_{\{v,u\}\in E_G} \mathrm{dist}_H \big( f(v), f(u)\big) \geqslant \eta\, |M_0'|\, \log_{\Delta-1}m,
\end{equation}
where $\eta$ is as in \eqref{e5.14}.
\end{claim}

\begin{proof}[Proof of Claim \ref{claim5.9}]
Notice that if $M_0'=\emptyset$, then the statement is trivial; therefore, we assume that $M_0'\neq\emptyset$. Assume, towards a contradiction, that
\begin{equation}\label{e5.18}
\sum_{\{v,u\}\in E_G} \mathrm{dist}_H \big( f(v), f(u)\big) < \eta \, |M_0'|\, \log_{\Delta-1}m.
\end{equation}
For every integer $k\geqslant 0$, set
\[ T_k := \big\{ v\in [n]\colon \mathrm{dist}_G(v,M_0')=k\big\} \ \ \ \text{ and } \ \ \
E_k := \big\{ e\in E_G \colon e\cap T_k \neq \emptyset \text{ and } e\cap T_{k+1}\neq\emptyset\big\};   \]
we also set
\[ k_0 := \min \bigg\{ k\geqslant 0 \colon \Big|\bigcup_{\ell=0}^k T_\ell\Big| \geqslant \frac{15n}{16} \bigg\}. \]
Let $k\in [k_0]$; then, since $T_k = \big\{v\in[n] \colon \mathrm{dist}_G\big(v, B_G(M_0',k-1)\big)=1\big\}$, by the choice of $k_0$, Lemma~\ref{lemma5.2} applied  for ``$\xi = 1/16$'' and ``$A = B_G(M_0', k-1)$'', and Lemma~\ref{lemma5.3}, we have
\begin{equation} \label{e5.19}
|T_k| \geqslant \frac{1}{3\cdot 10^3}\, \big|B_G(M_0',k-1)\big| \geqslant
\frac{\tilde{\alpha}}{3\cdot 10^3 (d-1)} \, (d-1)^k\, |M_0'|,
\end{equation}
where $\tilde{\alpha}$ is as in \eqref{e5.5}. For every $k\in\{0,\dots,k_0-1\}$, set
\[ E_k^{\mathrm{Atyp}} := \bigg\{ \{v,u\}\in E_k \colon
\mathrm{dist}_H\big(f(v), f(u)\big) \geqslant\frac{\varepsilon}{6 (d-1)^{\frac{k}{2}}} \log_{\Delta-1}m\bigg\}, \]
\[ E_k^{\mathrm{Typ}} := E_k\setminus E_k^{\mathrm{Atyp}}. \]
We shall refer to the edges in $E_k^{\mathrm{Atyp}}$ as \textit{atypical}. By \eqref{e5.18}, Markov's inequality and the choice of $\eta$ in \eqref{e5.14}, we see that
\begin{equation} \label{e5.20}
|E_k^{\mathrm{Atyp}}| \leqslant \frac{6\eta}{\varepsilon} \, (d-1)^{\frac{k}{2}}\, |M_0'| =
\frac{\tilde{\alpha}}{2^7\cdot 3\cdot 10^3}(d-1)^{\frac{k}{2}}|M_0'|.
\end{equation}
For every $k\in[k_0]$, set
\begin{align*}
T_k^{\mathrm{Typ}} := \Big\{ v\in T_k\colon \text{there exists a path of } &
\text{length $k$ from a vertex in } M_0' \text{ to } v\\
& \text{such that all its edges are in } \bigcup_{\ell=0}^{k-1} E_\ell^{\mathrm{Typ}}\Big\}.
\end{align*}
Let us check that
\begin{equation}\label{e5.21}
|T_k^{\mathrm{Typ}}| \geqslant \Big( 1- \frac{1}{2^5} \Big)|T_k|.
\end{equation}
First, notice that for every $v\in T_k$, there exists a path of length $k$ from $M_0'$ to $v$; and, if $v\notin T_k^{\mathrm{Typ}}$, then one of the edges of this path must be atypical. Using this observation, the $d$-regularity of $G$, and the fact that $\sum_{\ell=0}^{\infty}(d-1)^{-\frac{\ell}{2}}\leqslant 4$, the claimed inequality \eqref{e5.21} follows:
\begin{align*}
|T_k\setminus T_k^{\mathrm{Typ}}| & \leqslant \sum_{\ell=1}^k |E_{\ell-1}^{\mathrm{Atyp}}|\, (d-1)^{k-\ell}
\stackrel{\eqref{e5.20}}{\leqslant}
\frac{\tilde{\alpha}}{2^7\cdot 3\cdot 10^3} \sum_{\ell=1}^k (d-1)^{\frac{\ell-1}{2}} |M_0'|(d-1)^{k-\ell} \\
& \leqslant  \frac{\tilde{\alpha}}{2^7\cdot 3\cdot 10^3\cdot\sqrt{d-1}}\, (d-1)^k\, |M_0'|\,
\sum_{\ell=1}^k (d-1)^{-\frac{\ell}{2}} \\
& \leqslant \frac{\tilde{\alpha}}{2^5\cdot 3\cdot 10^3\cdot (d-1)}\, (d-1)^k\, |M_0'|
\stackrel{\eqref{e5.19}}{\leqslant} \frac{1}{2^5}\, |T_k|.
\end{align*}
Since $M_0', T_1,\dots, T_{k_0}, M_2(f,\varepsilon)$ are pairwise disjoint subsets of $[n]$ and $|M_2(f,\varepsilon)|\geqslant \frac{n}{8}$, by \eqref{e5.21}, there exists $k_*\in [k_0]$ such that, setting
\[ M_2^{\mathrm{Typ}} := M_2(f,\varepsilon)\cap T_{k_*}^{\mathrm{Typ}}, \]
we have
\begin{equation}\label{e5.22}
|M_2^{\mathrm{Typ}}| \geqslant \frac{1}{2^5} |T_{k_*}|.
\end{equation}
For every $v\in M_0'$, set
\[ T(v) :=\big\{ u\in T_{k_*}\colon \mathrm{dist}_G(v,u) = k_*\big\}; \]
also, set
\[ A := \big\{ v\in M_0' \colon T_{k_*}(v) \cap M_2^{\mathrm{Typ}} \neq\emptyset\big\}. \]
Using again the fact that $\sum_{k=1}^{\infty}(d-1)^{-\frac{k}{2}}\leqslant3$, we have for every $v\in A$ and every $u\in T(v)\cap M_2^{\mathrm{Typ}}$,
\[  \mathrm{dist}_H\big(f(v), f(u)\big) \leqslant
\frac{\varepsilon}{6} \log_{\Delta-1}m \sum_{k=1}^{\infty}(d-1)^{-\frac{k}{2}} \leqslant
\frac{\varepsilon}{2} \log_{\Delta-1}m. \]
Therefore,
\begin{equation}\label{e5.23}
f(A)\subseteq B_H \bigg( f(M_2^{\mathrm{Typ}}), \Big\lfloor \frac{\varepsilon}{2} \log_{\Delta-1}m \Big\rfloor\bigg)
\subseteq B_H \bigg( f\big( M_2(f,\varepsilon)\big), \Big\lfloor \frac{\varepsilon}{2} \log_{\Delta-1}m \Big\rfloor\bigg).
\end{equation}
Since $G$ is $d$-regular, for every $v\in M_0'$, we have $|T(v)|\leqslant d(d-1)^{k_*-1}\leqslant3(d-1)^{k_*}$.
Hence, by the definition of $A$ and inequalities \eqref{e5.19} and \eqref{e5.22}, we obtain that
\[ |A| \geqslant \frac{|M_2^{\mathrm{Typ}}|}{3(d-1)^{k_*}} \geqslant
\frac{\tilde{\alpha}}{2^5 \, 3^2 \, 10^3 \, (d-1)} |M_0'|. \]
By the definition of $M_0'$ in \eqref{e5.12}, we may select a positive integer $r$ with $r^*-2\log_2 m \leqslant r\leqslant r^*$ such that
\[ |A\cap M_{0,r}| \geqslant \frac{\tilde{\alpha}}{2^5\, 3^2\, 10^3\, (d-1)} |M_{0,r}| \ \ \
\text{ and } \ \ \ |f(M_{0,r})|\geqslant m^{1-2\varepsilon}, \]
where $M_{0,r}$ is as in \eqref{e5.10}. Consequently, $|M_{0,r}|\geqslant |f(M_{0,r})|\,2^{r-1}\geqslant m^{1-2\varepsilon}\,2^{r-1}$ which, in turn, implies that
\[ |f(A)| \geqslant|f(A\cap M_{0,r})| \geqslant\frac{|A\cap M_{0,r}|}{2^r}
\geqslant \frac{\tilde{\alpha}}{2^5\, 3^2\, 10^3\, (d-1)} \cdot \frac{|M_{0,r}|}{2^r}
\geqslant \frac{\tilde{\alpha}}{2^6\, 3^2\, 10^3\, (d-1) } m^{1-2\varepsilon}. \]
On the other hand, $\big|f\big(M_2(f,\varepsilon)\big)\big|\leqslant m^{1-4\varepsilon}$ and so, by \eqref{e5.23}, there exists $x\in f\big(M_2(f,\varepsilon)\big)$~such~that
\begin{equation} \label{e5.24}
\bigg|B_H \bigg( x, \Big\lfloor \frac{\varepsilon}{2} \log_{\Delta-1}m \Big\rfloor \bigg)\bigg| \geqslant
\frac{|f(A)|}{\big|f\big(M_2(f,\varepsilon)\big)\big|} \geqslant
\frac{\tilde{\alpha}}{2^6\, 3^2\, 10^3\, (d-1)} m^{2\varepsilon}.
\end{equation}
However, by the $\Delta$-regularity of $H$, we have
\begin{equation}\label{e5.25}
\bigg|B_H \bigg( x,\Big\lfloor \frac{\varepsilon}{2} \log_{\Delta-1}m \Big\rfloor \bigg)\bigg| \leqslant
3 m^{\frac{\varepsilon}{2}}.
\end{equation}
If $m$ is sufficiently large in terms of $d, \alpha$ and $\varepsilon$, then, by \eqref{e5.24} and \eqref{e5.25}, we derive a contradiction. The proof of Claim \ref{claim5.9} is thus completed.
\end{proof}

The following claim is the second and final step in the proof of Lemma \ref{l5.8}.

\begin{claim} \label{claim5.10}
We have
\begin{equation} \label{e5.27}
\sum_{\substack{\{v,u\}\in E_G \\\{v,u\}\cap M_0''\neq \emptyset}} \mathrm{dist}_H \big( f(v), f(u)\big) \leqslant
84 \cdot 10^2 \cdot \frac{\log_{\Delta-1}m}{m} \cdot \sum_{\{v,u\}\in E_G} \mathrm{dist}_H \big( f(v), f(u)\big)
\end{equation}
and
\begin{equation} \label{e5.28}
\sum_{\substack{v,u\in [n] \\ \{v,u\}\cap M_0''\neq \emptyset}} \mathrm{dist}_H \big( f(v), f(u)\big) \leqslant
48\cdot \frac{\log_{\Delta-1}m}{m} \cdot \sum_{v,u\in [n]} \mathrm{dist}_H \big( f(v), f(u)\big),
\end{equation}
where $M_0''$ is as in \eqref{e5.12}.
\end{claim}

\begin{proof}[Proof of Claim \ref{claim5.10}]
Clearly, we may assume that $M_0''\neq\emptyset$. Set
\[ E_0:= \big\{ e\in E_G \colon e\cap M_0''\neq \emptyset\big\} \ \ \ \text{ and } \ \ \
E_1:= \big\{e\in E_G \colon e\cap M_{0,r^*}\neq\emptyset \text{ and } e\setminus M_{0,r^*}\neq \emptyset\big\}. \]
Since $M_{0,r^*},M_2(f,\varepsilon)$ are disjoint and $|M_2(f,\varepsilon)|\geqslant \frac{n}{8}$, by Lemma \ref{lemma5.2} applied for ``$\xi=1/8$'' and ``$A=M_{0,r^*}$'', we obtain that
\begin{equation}\label{e5.29}
|E_1|\geqslant \frac{d}{2\cdot 7\cdot 10^2}|M_{0,r^*}|.
\end{equation}
Moreover, as $|M_{0,r^*}|\geqslant 2^{r^*-1}$ and $\sum_{r=1}^{\lfloor r^*-2\log_2m\rfloor}|f(M_{0,r})|\leqslant |\mathrm{Im}f|\leqslant m$, we have that
\begin{equation} \label{e5.30}
|M_0''|\leqslant \sum_{r=1}^{\lfloor r^*-2\log_2m\rfloor} 2^r\, |f(M_{0,r})| \leqslant
\frac{2^{r^*}}{m} \leqslant \frac{2\,|M_{0,r^*}|}{m}.
\end{equation}
By part~(\hyperref[Part-B-R(e)]{$B$}) of property~$\mathcal{R}(\varepsilon)$, the definition of $E_0$, and the $d$-regularity of $G$,
\begin{align}\label{e5.31}
\sum_{\{v,u\}\in E_0}\mathrm{dist}_H\big(f(v), f(u)\big) & \ \leqslant |E_0|\, 3\, \log_{\Delta-1}m \leqslant
d\, |M_0''|\, 3\, \log_{\Delta-1}m \\
& \stackrel{\eqref{e5.30}}{\leqslant} \frac{6\,d\, \log_{\Delta-1}m}{m} \, |M_{0,r^*}|
\stackrel{\eqref{e5.29}}{\leqslant} 84\cdot 10^2\cdot \frac{\log_{\Delta-1}m}{m}\, |E_1|. \nonumber
\end{align}
Since $\mathrm{dist}_H\big(f(v), f(u)\big)\geqslant1$ for every $\{v,u\}\in E_1$, we see that \eqref{e5.27} follows from \eqref{e5.31}.

Finally, using again part~(\hyperref[Part-B-R(e)]{$B$}) of property~$\mathcal{R}(\varepsilon)$ and the fact that $|M_2(f,\varepsilon)|\geqslant \frac{n}{8}$, we have
\begin{equation}\label{e5.32}
\sum_{\substack{v,u\in [n] \\ \{v,u\}\cap M_0''\neq \emptyset}} \mathrm{dist}_H \big( f(v), f(u)\big)
\leqslant 6\, |M_0''|\, n \, \log_{\Delta-1}m
\stackrel{\eqref{e5.30}}{\leqslant}
96\cdot \frac{\log_{\Delta-1}m}{m} \cdot |M_{0,r^*}| \cdot |M_2(f,\varepsilon)|.
\end{equation}
Observing that $\mathrm{dist}_H\big(f(v), f(u)\big)\geqslant1$ for every $(v,u)\in M_{0,r^*}\times M_2(f,\varepsilon)$, we see that \eqref{e5.28} follows from \eqref{e5.32}. The proof of Claim \ref{claim5.10} is completed.
\end{proof}

We are now in a position to complete the proof of Lemma \ref{l5.8}.

\begin{proof}[Proof of Lemma \ref{l5.8}]
Notice that, almost everywhere, $\hat f$ and $\boldsymbol{f}$ differ only on $M_0'$. Hence, by the $d$-regularity of~$G$ and part~(\hyperref[Part-B-R(e)]{$B$}) of property~$\mathcal{R}(\varepsilon)$, we have, almost everywhere,
\begin{align*}
\sum_{\{v,u\}\in E_G}\mathrm{dist}_H\big( \boldsymbol{f}(v), \boldsymbol{f}(u)  \big)
& \leqslant \sum_{\{v,u\}\in E_G}\mathrm{dist}_H\big( \hat f(v), \hat f(u)  \big)
+\sum_{\substack{\{v,u\}\in E_G \\ \{v,u\}\cap M_0'\neq\emptyset}}
\mathrm{dist}_H\big( \boldsymbol{f}(v), \boldsymbol{f}(u)  \big)\\
& \, \leqslant \sum_{\{v,u\}\in E_G}\mathrm{dist}_H\big( \hat f(v), \hat f(u)  \big)
+ 3\, d\, |M_0'|\, \log_{\Delta-1}m.
\end{align*}
Consequently, by Claim \ref{claim5.9}, if $m\geqslant m_4$, then we have, almost everywhere,
\begin{equation}\label{e5.33}
\sum_{\{v,u\}\in E_G}\mathrm{dist}_H\big( \boldsymbol{f}(v), \boldsymbol{f}(u)  \big)
\leqslant \sum_{\{v,u\}\in E_G}\mathrm{dist}_H\big( \hat f(v), \hat f(u)  \big)
+\frac{3d}{\eta}\, \sum_{\{v,u\}\in E_G} \mathrm{dist}_H \big( f(v), f(u)\big).
\end{equation}
On the other hand, observe that $\hat f$ and $f$ differ only on $M_0''$ and $\hat f$ is constant on $M_0''$. Therefore, by part~(\hyperref[Part-B-R(e)]{$B$}) of property~$\mathcal{R}(\varepsilon)$ and the fact that $\mathrm{dist}_H\big( f(v), f(u)\big)\geqslant 1$ if $v\in M_0''$ and $u\not\in M_0''$,
\begin{align*}
\sum_{\{v,u\}\in E_G}\mathrm{dist}_H\big( \hat f(v), \hat f(u)  \big)
& \leqslant \sum_{\{v,u\}\in E_G}\mathrm{dist}_H\big( f(v), f(u)  \big)
+ \sum_{\substack{\{v,u\}\in E_G\\ v\in M_0'', u\not\in M_0''}}\mathrm{dist}_H\big( \hat f(v), \hat f(u)  \big)\\
& \leqslant \sum_{\{v,u\}\in E_G}\mathrm{dist}_H\big( f(v), f(u)  \big)
+ 3\, \log_{\Delta-1}m\, \sum_{\substack{\{v,u\}\in E_G\\ v\in M_0'', u\not\in M_0''}}\mathrm{dist}_H\big( f(v),  f(u)  \big),
\end{align*}
which further implies, by \eqref{e5.27} and the fact that $\Delta\geqslant 3$,
\begin{align} \label{e5.34}
\sum_{\{v,u\}\in E_G}\mathrm{dist}_H\big( \hat f(v), \hat f(u)  \big)
& \leqslant \Big(1+252 \cdot 10 ^2\cdot  \frac{\log^2_{\Delta-1}m}{m}\Big)
\sum_{\{v,u\}\in E_G}\mathrm{dist}_H\big( f(v), f(u)  \big) \\
& \leqslant \Big(1+252 \cdot 10 ^2\cdot  \frac{\log^2_2 m}{m}\Big)
\sum_{\{v,u\}\in E_G}\mathrm{dist}_H\big( f(v), f(u)  \big). \nonumber
\end{align}
Since $d/\eta\geqslant 2$, we see that there exists a positive integer $m_5=m_5(d,\alpha,\varepsilon)$ such that \eqref{e5.15} follows from \eqref{e5.33} and \eqref{e5.34} for every integer $m\geqslant m_5$.

We proceed to the proof of \eqref{e5.16}. Again observe that $f$ and $\hat f$ differ only on $M_0''$, and that $\hat f$ is constant on $M_0''$. Therefore, by \eqref{e5.28} and the fact that $\Delta\geqslant 3$,
\begin{align*}
\sum_{v,u\in [n]} \mathrm{dist}_H\big( f(v), f(u)  \big)
& = \sum_{v,u\in [n]} \mathrm{dist}_H\big(\hat f(v),\hat f(u)  \big)
+ \sum_{\substack{v,u\in [n] \\ \{v,u\}\cap M_0''\neq\emptyset}} \mathrm{dist}_H\big( f(v), f(u)  \big)\\
& \leqslant\sum_{v,u\in [n]} \mathrm{dist}_H\big(\hat f(v),\hat f(u)  \big)
+ 48\cdot  \frac{\log_{\Delta-1}m}{m} \cdot \sum_{v,u\in [n]} \mathrm{dist}_H \big( f(v), f(u)\big) \\
& \leqslant\sum_{v,u\in [n]} \mathrm{dist}_H\big(\hat f(v),\hat f(u)  \big)
+ 48\cdot  \frac{\log_2 m}{m} \cdot \sum_{v,u\in [n]} \mathrm{dist}_H \big( f(v), f(u)\big).
\end{align*}
Thus, there exists an absolute positive integer $m_6$ such that for every integer $m\geqslant m_6$,
\begin{equation}\label{e5.37}
\frac{1}{2} \sum_{v,u\in [n]} \mathrm{dist}_H\big( f(v), f(u)  \big)
\leqslant \sum_{v,u\in [n]} \mathrm{dist}_H\big(\hat f(v),\hat f(u) \big).
\end{equation}
Next observe that
\[  \sum_{v\in M_0', u\not\in M_0'} \!\! \mathbb{E} \big[ \mathrm{dist}_H\big( \boldsymbol{f}(v),\boldsymbol{f}(u)\big) \big]
= \frac{1}{|M_0'|}\sum_{v,w\in M_0', u\not\in M_0'} \!\! \mathrm{dist}_H\big( \hat f(w), \hat f(u)  \big)
= \!\! \sum_{v\in M_0', u\not\in M_0'} \!\! \mathrm{dist}_H\big( \hat f(v), \hat f(u)  \big) \]
and
\[ \sum_{v,u\in M_0'}\mathbb{E}\big[\mathrm{dist}_H\big( \boldsymbol{f}(v), \boldsymbol{f}(u)  \big) \big]=0; \]
therefore,
\begin{equation}\label{e5.38}
  \sum_{v,u\in[n]} \mathbb{E}\big[\mathrm{dist}_H\big( \boldsymbol{f}(v), \boldsymbol{f}(u)  \big) \big]
  =  \sum_{v,u\in[n]} \mathrm{dist}_H\big( \hat f(v), \hat f(u)  \big)
  - \sum_{v,u\in M_0'} \mathrm{dist}_H\big( \hat f(v), \hat f(u)  \big).
\end{equation}
Since $n- |M_0'|\geqslant |M_2(f,\varepsilon)| \geqslant \frac{n}{8}$, we see that $\frac{|M_0'|}{n-|M_0'|} \leqslant 7$. Thus, by triangle inequality,
\begin{align*}
\sum_{v,u\in M_0'} \mathrm{dist}_H\big( \hat f(v), \hat f(u)  \big)
& \leqslant \frac{1}{n-|M_0'|} \sum_{v,u\in M_0',w\not\in M_0'}
\Big(\mathrm{dist}_H\big( \hat f(v), \hat f(w)  \big) +\mathrm{dist}_H\big( \hat f(w), \hat f(u) \big) \Big) \\
& \leqslant 2\, \frac{|M_0'|}{n-|M_0'|} \sum_{v\in M_0',u\not\in M_0'} \mathrm{dist}_H\big( \hat f(v), \hat f(u) \big)\\
&  \leqslant 7 \Big( \sum_{v,u\in [n]} \mathrm{dist}_H\big( \hat f(v), \hat f(u)  \big) -
\sum_{v,u\in M_0'} \mathrm{dist}_H\big( \hat f(v), \hat f(u)  \big)\Big),
\end{align*}
that implies that,
\begin{equation} \label{e5.39}
\sum_{v,u\in [n]} \mathrm{dist}_H\big( \hat f(v), \hat f(u)  \big)
\leqslant 8 \Big( \sum_{v,u\in [n]} \mathrm{dist}_H\big( \hat f(v), \hat f(u)  \big) -
\sum_{v,u\in M_0'} \mathrm{dist}_H\big( \hat f(v), \hat f(u)  \big)\Big).
\end{equation}
The desired estimate \eqref{e5.16} follows from \eqref{e5.37}--\eqref{e5.39} for every integer $m\geqslant m_6$. The proof of Lemma \ref{l5.8} is completed.
\end{proof}

\section{Proofs of the main results} \label{sec6}

Throughout this section, $d\geqslant 3$ is a fixed integer.

\subsection{Initializing various parameters} \label{subsec6.1}

We start by setting
\begin{align}
\label{e6.1} \alpha=\alpha(d) & := e^{-10^{11}\log^2 d}, \\
\label{e6.2} \varepsilon & := 10^{-4};
\end{align}
moreover, denoting by $K=K(d)>0$ the constant hidden by the big-O notation in the right-hand side of \eqref{e9}, we set
\begin{equation} \label{e6.3}
m_7=m_7(d):= \max\Big\{ m_0(d),m_1(d,\alpha,\varepsilon), \frac{1}{K}\Big\},
\end{equation}
where $m_0(d)$ and $m_1(d,\alpha,\varepsilon)$ are as in Proposition \ref{small-M2} and Proposition \ref{large-M2}, respectively, for the choices of $\alpha$ and $\varepsilon$ in \eqref{e6.1} and \eqref{e6.2}.

Next, let $n_0(d)$ denote the least positive integer such that for every integer $n\geqslant n_0(d)$, with probability at least $\frac12$, a~uniformly random graph $G\in G(n,d)$ satisfies property $\mathcal{D}(\alpha)$, where $\alpha$ is as in \eqref{e6.1}. (Notice that, by Proposition \ref{typical}, $n_0(d)$ is well-defined.) Finally, we set
\begin{equation} \label{e6.4}
n_1=n_1(d):= \max\big\{ 4d, 200, n_0(d) \big\}.
\end{equation}

\subsection{Definition of ``good" events} \label{susec6.2}

Let $n\geqslant n_1$, $\Delta\geqslant 3$ and $m\geqslant m_7$ be integers, and set
\begin{align}
\label{e6.5} & \ \ \mathcal{D}_n := \big\{G\in G(n,d)\colon G \text{ satisfies property } \mathcal{D}(\alpha)\big\}, \\
\label{e6.6} & \mathcal{R}_{m,\Delta} := \big\{H\in G(m,\Delta)\colon H \text{ satisfies property } \mathcal{R}(\varepsilon)\big\},
\end{align}
where $\alpha$ and $\varepsilon$ are as in \eqref{e6.1} and \eqref{e6.2}, respectively. We also set
\[ \mathbb{P}_n[\, \cdot \,]:= \mathbb{P}_{G(n,d)}\big[\, \cdot \, \big| \, \mathcal{D}_n\big] \ \ \ \text{ and } \ \ \
\mathbb{P}_{m,\Delta}[\, \cdot \,]:= \mathbb{P}_{G(m,\Delta)}\big[\, \cdot \, \big| \, \mathcal{R}_{m,\Delta}\big], \]
and we define
\begin{equation} \label{e6.7}
\mathcal{E}_{n,m,\Delta}:= \Big\{ (G,H)\in \mathcal{D}_n\times \mathcal{R}_{m,\Delta}\colon
\gamma(G,\dist_H)\leqslant \ \max\big\{\Gamma_1(\varepsilon),\Gamma_2(d,\alpha,\varepsilon)\big\} \Big\},
\end{equation}
where $\Gamma_1(\varepsilon)$ and $\Gamma_2(d,\alpha,\varepsilon)$ are as in \eqref{e8} and \eqref{e10}, respectively, for the choices of $\alpha$ and $\varepsilon$ in \eqref{e6.1} and \eqref{e6.2}.

We will need the following corollary of Propositions \ref{small-M2} and \ref{large-M2}. The proof is straightforward.

\begin{corollary} \label{cor6.1}
For every triple $n\geqslant n_1$, $\Delta\geqslant 3$ and $m\geqslant m_7$ of integers, we have
\begin{equation} \label{e6.8}
\mathbb{P}_n \times \mathbb{P}_{m,\Delta}\big[ \mathcal{E}_{n,m,\Delta}\big] \geqslant
1- \frac{1}{m^{0.04n}}.
\end{equation}
\end{corollary}

\subsection{Proof of Theorem \ref{thm-kleinberg}} \label{subsec6.3}

By Corollary \ref{cor6.1}, Theorem \ref{extrapolation} and \eqref{e-cheeger-e1}, for every triple $n\geqslant n_1$, $\Delta\geqslant 3$ and $m\geqslant m_7$ of integers, we have
\[ \mathbb{P}_n \times \mathbb{P}_{m,\Delta}\big[(G,H)\colon \gamma(G,\dist_H^p)\leqslant \Gamma(d,p)
\text{ for every } p\geqslant 1\big] \geqslant
1- \frac{1}{m^{0.04n}}, \]
where $\Gamma(d,p)$ is as in \eqref{klein-e1}. The result follows from this estimate and Proposition \ref{typical}.

\subsection{Proof of Corollary \ref{cor-bi-Lip}} \label{subsec6.4}

By Theorem \ref{thm-kleinberg}, it is enough to show that $c_H(G)\gtrsim_d \log n$ for every $G\in G(n,d)$ and every $H\in G(m,\Delta)$ with $\gamma(G,\dist_H)\leqslant \Gamma(d,1)$, where $\Gamma(d,1)$ is as in \eqref{klein-e1}. To this end, let $e\colon [n]\to [m]$ be any map such that $s\,\dist_G(v,u)\leqslant \dist_H\big((e(v),e(u)\big)\leqslant sD\, \dist_G(v,u)$ for all $v,u\in [n]$,
where $s,D>0$. Then, by the $d$-regularity of $G$, we obtain that
\begin{align*}
\log n & \lesssim_d \frac{1}{n^2} \sum_{v,u\in [n]}\dist_G(v,u) \leqslant
\frac{1}{s} \cdot \frac{1}{n^2} \sum_{v,u\in [n]}\dist_H\big(e(v),e(u)\big) \\
& \leqslant \frac{\Gamma(d,1)}{s} \cdot \frac{1}{|E_G|} \sum_{\{v,u\}\in E_G}\dist_H\big(e(v),e(u)\big)
\leqslant \Gamma(d,1)\cdot D \stackrel{\eqref{klein-e1}}{\lesssim_d} D,
\end{align*}
that yields that $c_H(G)\gtrsim_d \log n$.

\subsection{Proof of Theorem \ref{mn-random}} \label{subsec6.5}

It is based on the following lemma.

\begin{lemma} \label{lem6.2}
For every integer $n\geqslant n_1$, there exists a subset $\mathcal{A}_n$ of\, $G(n,d)$ with
\begin{equation} \label{e6.9}
\mathbb{P}_n\big[\mathcal{A}_n\big] \geqslant 1-\frac{1}{2^{n/200}}
\end{equation}
such that for every $G\in\mathcal{A}_n$, every integer $\Delta\geqslant 3$ and every integer $m\geqslant m_7$, we have
\begin{equation} \label{e6.10}
\mathbb{P}_{m,\Delta}\big[H \colon \gamma(G,\dist_H^p)\leqslant \Gamma(d,p)
\text{ for every } p\geqslant 1 \big] \geqslant 1- \frac{1}{m^{0.02n}},
\end{equation}
where $\Gamma(d,p)$ is as in is as in \eqref{klein-e1}.
\end{lemma}

Granting Lemma \ref{lem6.2}, the proof of Theorem \ref{mn-random} follows by setting $C(d):=n_1$ and using a union bound, the fact that for every integer $m\geqslant 3$ we have
\[ \sum_{n\geqslant n_1} \frac{1}{m^{0.02n}}=O\Big(\frac{1}{m}\Big), \]
and Proposition \ref{typical}.

So, it remains to prove Lemma \ref{lem6.2}. To this end, we observe that, by Corollary \ref{cor6.1} and a double counting argument, for every triple $n\geqslant n_1$, $\Delta\geqslant 3$ and $m\geqslant m_7$ of integers, there exists a subset $\mathcal{A}_n^{m,\Delta}$ of $G(n,d)$ with
\begin{equation} \label{e6.11}
\mathbb{P}_n\big[\mathcal{A}_n^{m,\Delta}\big] \geqslant 1-\frac{1}{m^{0.02n}}
\end{equation}
such that for every $G\in\mathcal{A}_n^{m,\Delta}$, we have
\begin{equation} \label{e6.12}
\mathbb{P}_{m,\Delta}\big[H\colon (G,H)\in \mathcal{E}_{n,m,\Delta}\big] \geqslant 1-\frac{1}{m^{0.02n}}.
\end{equation}
For every integer $n\geqslant n_1$, we set
\begin{equation} \label{e6.13}
\mathcal{A}_n := \bigcap_{m\geqslant m_7} \bigcap_{\substack{3 \leqslant \Delta\leqslant m \\ \Delta m \text{ even}}}
\mathcal{A}_n^{m,\Delta}
\end{equation}
and we observe that
\begin{equation} \label{e6.14}
\mathbb{P}_n\big[\mathcal{A}_n\big] \geqslant 1-\sum_{m\geqslant m_7}
\sum_{\substack{3 \leqslant \Delta\leqslant m \\ \Delta m \text{ even}}}
\frac{1}{m^{0.02n}} \geqslant 1-\frac{1}{2^{n/200}}.
\end{equation}
Next, fix $G\in\mathcal{A}_n$, and let $\Delta\geqslant 3$ and $m\geqslant m_7$ be arbitrary. Since $G\in \mathcal{A}_n\subseteq \mathcal{A}_n^{m,\Delta}$, by the definition of $\mathcal{E}_{n,m,\Delta}$ in \eqref{e6.7}, the estimate in \eqref{e6.12} and the choice of $\Gamma(d,1)$ in \eqref{klein-e1}, we see that
\begin{equation} \label{e6.15}
\mathbb{P}_{m,\Delta}\big[H \colon \gamma(G,\dist_H)\leqslant \Gamma(d,1) \big] \geqslant 1- \frac{1}{m^{0.02n}}.
\end{equation}
After observing that every graph in $\mathcal{A}_n$ satisfies property $\mathcal{R}(\alpha)$, where $\alpha$ is as in \eqref{e6.1}, the desired estimate \eqref{e6.10} follows from \eqref{e6.15}, Theorem \ref{extrapolation} and \eqref{e-cheeger-e1}. This completes the proof of Lemma~\ref{lem6.2}, and so the proof of Theorem \ref{mn-random} is also completed.

\subsection{Proof of Theorem \ref{universal-approximator}} \label{subsec6.6}

The proof is identical to that of \cite[Theorem 2.1]{MN15} with the main new ingredient being Theorem \ref{mn-random}. We recall the argument for the convenience of the reader.

Let $C(3)$ be the integer obtained by Theorem \ref{mn-random} applied for $d=3$ and, for every even integer $2k\geqslant C(3)$, we fix $G_{2k}\in G(2k,3)$ such that, for every integer $\Delta\geqslant 3$,
\begin{equation} \label{e6.16}
\mathbb{P}_{G(m,\Delta)} \bigg[H\colon \sup_{k\geqslant C(3)} \gamma(G_{2k},\dist_H^p)\leqslant \Gamma(p)
\text{ for every } p\geqslant 1 \bigg] \geqslant 1-O_{\Delta}\Big(\frac{1}{m^{\tau''}}\Big),
\end{equation}
where $\tau''=\tau''(\Delta)$ is as in Theorem \ref{mn-random}.

For every positive integer $k$, we define a multi-graph $U_k$ on $[k]$ as follows. If $k<C(3)$, then let $U_k$ be the complete graph on $[k]$ vertices. Next assume that $k\geqslant C(3)$; we partition the set $[2k]$ into disjoint set $A_1,\dots,A_k$ each of size $2$ and, for every $i,j\in [k]$, we put the edge (or the self-loop) $\{i,j\}$ in $E_{U_k}$ for each $\{v,u\}\in E_{G_{2k}}$ with $v\in A_i$ and~$u\in A_j$. Since $G_{2k}$ is $3$-regular, we see that $|E_{U_k}|=|E_{G_{2k}}|=3k$.

Let $\Delta\geqslant 3$ be an integer, and let $H$ be a $\Delta$-regular graph such that for every $p\geqslant 1$,
\begin{equation} \label{uni-e1}
\sup_{k\geqslant C(3)} \gamma(G_{2k},\dist_H^p)\leqslant \Gamma(p).
\end{equation}
By \eqref{e6.16}, it is enough to show that, for every $p\geqslant 1$ and every positive integer $k$, the multi-graph $U_k$ is a $\big(2^p\Gamma(p)\big)$-universal approximator of $\dist_H^p$. Notice that if $k<C(3)$, then this follows automatically by the choice of $U_k$; therefore, we may assume that $k\geqslant C(3)$. Fix $x_1,\dots,x_k\in V_H$, and define $f\colon [2k]\to V_H$ by setting $f(v)=x_i$ if $v\in A_i$. Then observe that
\begin{align} \label{uni-e2}
\frac{1}{k^2} \sum_{i,j=1}^k \dist_H(x_i,x_j)^p & \ = \frac{1}{(2k)^2} \sum_{v,u\in [2k]}  \dist_H\big(f(v),f(u)\big)^p \\
& \stackrel{\eqref{uni-e1}}{\leqslant} \frac{\Gamma(p)}{|E_{G_{2k}}|} \sum_{\{v,u\}\in E_{G_{2k}}}
\dist_H\big(f(v),f(u)\big)^p \nonumber \\
& \ \leqslant \frac{2^{p+1}\Gamma(p)}{(2k)^2}\!\! \sum_{v,u\in [2k]} \!\! \dist_H\big( f(v),f(u)\big)^p =
\frac{2^{p+1}\Gamma(p)}{k^2} \sum_{i,j=1}^k \dist_H(x_i,x_j)^p, \nonumber
\end{align}
where the last inequality in \eqref{uni-e2} follows by averaging the elementary estimate
\[ \dist_H\big(f(v),f(u)\big)^p\leqslant 2^{p-1} \dist_H\big(f(v),f(w)\big)^p +
2^{p-1} \dist_H\big(f(w),f(u)\big)^p \]
over the set $\big\{(v,u,w)\in [2k]^3\colon \{v,u\}\in E_{G_{2k}}\big\}$ and using the $3$-regularity of $G_{2k}$. Moreover,
\begin{equation} \label{uni-e3}
\frac{\Gamma(p)}{|E_{G_{2k}}|} \sum_{\{v,u\}\in E_{G_{2k}}} \dist_H\big(f(v),f(u)\big)^p =
\frac{\Gamma(p)}{|E_{U_k}|} \sum_{\{i,j\}\in E_{U_k}} \dist_H(x_i,x_j)^p.
\end{equation}
The result follows from \eqref{uni-e2} and \eqref{uni-e3}.

\begin{remark}
The proof of Theorem \ref{universal-approximator} yields that $|E_{U_k}|\leqslant \max\big\{3k,C(3)^2\big\}$ for every positive integer $k$.
\end{remark}

\subsection*{Acknowledgments}

We thank the authors of \cite{EMN25a} for sharing their preprint after our work was posted on arXiv. We also thank Assaf Naor for many valuable suggestions.

The research was supported in the framework of H.F.R.I call ``Basic research Financing (Horizontal support of all Sciences)" under the National Recovery and Resilience Plan ``Greece 2.0" funded by the European Union--NextGenerationEU (H.F.R.I. Project Number: 15866).
The third named author (K.T.) is partially supported by the NSF grant DMS 2331037.

\end{document}